\documentclass[12pt]{article}

\usepackage{latexsym}
\usepackage{graphics}
\usepackage{xspace}
\usepackage{psfrag}
\usepackage{epsfig}
\usepackage{pst-all}
\usepackage{amssymb, amsmath, amsthm, verbatim, mathrsfs}
\usepackage{bbm}

\author{\sf{David Gamarnik} \thanks {Operations Research Center and Sloan School of Management, MIT, Cambridge, MA, 02139, e-mail: \tt{gamarnik@mit.edu} } \and 
\sf{Sidhant Misra} \thanks{Department of Electrical Engineering and Computer Science, MIT, Cambridge, MA, 02139, e-mail: \tt{sidhant@mit.edu}} }

\newtheorem{thm}{Theorem}
\newtheorem{lem}{Lemma}
\newtheorem{cor}{Corollary}
\theoremstyle{definition}
\newtheorem{defin}{Definition}

\newtheorem{assumption}{Assumption}

\begin{document}

\title{Giant Component in Random Multipartite Graphs with Given Degree Sequences}
\maketitle

\abstract
We study the problem of the existence of a giant component in a random multipartite graph. We consider a random multipartite graph with $p$ parts generated according to a given degree sequence $n_i^{\mathbf{d}}(n)$
which denotes the number of vertices in part $i$ of the multipartite graph with degree given by the vector $\mathbf{d}$. We assume that the empirical distribution of the 
degree sequence converges to a limiting probability distribution. Under certain mild
 regularity assumptions, we characterize the conditions
under which, with high probability, there exists a component of linear size. The characterization involves checking whether the Perron-Frobenius norm of the
matrix of means of a certain associated edge-biased distribution is greater than unity. We also specify the size of the giant component when it exists. 
We use the exploration process of Molloy and Reed to analyze the size of components in the random graph. The main challenges arise due to the multidimensionality of the random processes involved which prevents us
from directly applying the techniques from the standard unipartite case. 
In this paper we use techniques from the theory of multidimensional Galton-Watson processes along with Lyapunov function technique to overcome the challenges.
\begin{section}{Introduction}

	The problem of the existence of a giant component in random graphs was first studied by Erd\"os and R\'enyi. In their classical paper \cite{ErdosRenyi}, they considered a random graph model on $n$ and $m$ edges where each such possible graph is equally likely. They showed that if $m/n > \frac{1}{2} + \epsilon$, with high probability as $n \rightarrow \infty$ there exists a component of size linear in $n$ 
	in the random graph and that the size of this component as a fraction of $n$ converges to a given constant.
	
	 The degree distribution of the classical  Erd\"os-R\'enyi random graph has Poisson tails. However in many applications the degree distribution associated with an underlying graph does not satisfy this. For example,
	 many so-called ``scale-free" networks exhibit power law distribution of degrees. This motivated the study of random graphs generated according to a given degree sequence. The giant component problem
	on a random graph generated according to a given degree sequence was considered by Molloy and Reed \cite{MolloyReed1}. They provided conditions on the degree distribution under which a giant component 
	exists with high probability. Further in \cite{MolloyReed2}, they also showed that the size of the giant component as a fraction of the number of vertices converges in probability to a given positive constant.
	They used an exploration process to analyze the components of vertices of the random graph to prove their results. Similar results were established by Janson and Luczak in \cite{JansonLuczak} using 
	different techniques based on the convergence of empirical distributions of independent random variables. There have been several papers that have proved similar results with similar but different assumptions and 
	tighter error bounds \cite{Molloy12}, \cite{BollobasRiordan12}, \cite{Riordan12}.
	Results for the critical phase for random graphs with given degree sequences were derived by Kang and Seierstad in \cite{KangSeierstad08}.
	  All of these results consider a random graph on $n$ vertices with a given degree sequence where the distribution is uniform among all
	feasible graphs with the given degree sequence. The degree sequence is then assumed to converge to a probability distribution and the results provide conditions on this probability distribution for which a 
	giant component exists with high probability. 
	
	In this paper, we consider random \emph{multipartite} graphs  with $p$ parts with given degree distributions. Here $p$ is a fixed positive integer.
	Each vertex is associated with a degree vector $\mathbf{d}$, where each of its component $d_i, i \in [p]$
	dictates the number of 
	neighbors of the vertex in the corresponding part $i$ of the graph. As in previous papers, we assume that the empirical distribution associated with the number of vertices of degree $\mathbf{d}$ converges
	to a probability distribution. We then pose the problem of finding conditions under which there exists a giant component in the random graph with high probability. 
	Our approach is based on the analysis of the Molloy and Reed exploration process. The major bottleneck is that the exploration process is a multidimensional process and the techniques of Molloy and Reed
	of directly underestimating the exploration process by a one dimensional random walk does not apply to our case. 
	In order to overcome this difficultly, we construct a linear Lyapunov function based on the Perron-Frobenius theorem, a technique often used in the study of multidimensional branching processes. 
	Then we carefully couple the exploration process with some underestimating process to prove our results The coupling construction is also more involved due to the multidimensionality of the process. 
	This is because in contrast to the unipartite case, there are multiple types of clones (or half-edges) involved in the exploration process, corresponding to which pair of parts of the multipartite graph they belong to. 
	At every step of the exploration process, revealing the neighbor of such a clone leads to the addition 
	of clones of \emph{several} types to the component being currently explored. The particular numbers and types
	of these newly added clones is also dependent on the kind of clone whose neighbor was revealed.  So, the underestimating  process needs to be constructed in a way such that it simultaneously underestimates
	the exploration process for each possible type of clone involved. We do this by choosing the parameters of the underestimating process such that for each type of clone, the vector of additional clones which are added by
	revealing its neighbor is always component wise smaller than the same vector for the exploration process.
	
	All results regarding giant components typically use a configuration model corresponding to the given 
	degree distribution by splitting vertices into clones and performing a uniform matching of the clones. In the standard unipartite case, at every step of the exploration process all available clones can be treated same.
	However in the multipartite case, this is not the case. For example,
	 the neighbor of a vertex in part $1$ of the graph with degree $\mathbf{d}$ can lie in part $j$ only if $d_j > 0$. Further, this neighbor must also have a degree
	$\hat{\mathbf{d}}$ such that $\hat{d}_i > 0$. This poses the issue of the graph breaking down into parts with some of the $p$ parts of the graph getting disconnected from the others. To get past this we make a certain
	irreducibility assumption which we will carefully state later. This assumption not only addresses the above problem, but also enables us to construct linear Lyapunov functions by using the Perron-Frobenius theorem
	for irreducible non-negative matrices. 
	We also prove that with the irreducibility assumption,  the giant component when it exists is unique and has linearly many vertices in each of the $p$ parts of the graph.
	In \cite{BollobasRiordan12}, Bollobas and Riordan show that the existence and the size of the giant component in the \emph{unipartite} case 
	is closely associated with an \textit{edge-biased} branching process. In this paper, we also construct
	an analogous edge-biased branching process which is now a multi-type branching process, and prove similar results.
	
	Our study of random multipartite graphs is motivated by the fact that several real world networks naturally demonstrate a multipartite nature.
	The author-paper network, actor-movie network, the network of company ownership,
	the financial contagion model, heterogenous social networks, etc. are all multipartite \cite{Newmann99}, \cite{Boss04}, \cite{Jackson08}.
	 Examples of biological networks which exhibit multipartite structure include drug target networks, protein-protein interaction networks and human disease networks \cite{Goh07}, \cite{Yildrim07}, \cite{Morrison06}.
	 In many cases evidence suggests that explicitly modeling the multipartite structure results in more accurate models and predictions.
	 
	Random bipartite graphs ($p=2$) with given degree distributions were considered by Newmann et. al in \cite{Newmann01}. They used generating function heuristics to identify the critical point in the bipartite case. However, 
	they did not provide rigorous proofs of the result. Our result establishes a rigorous proof of this result and we show that in the special case $p=2$, the conditions we derive is equivalent to theirs.
	
	The rest of the paper is structured as follows.  In Section \ref{sec:def}, we start by introducing the basic definitions and the notion of a degree distribution for multipartite graphs. In Section \ref{sec:results}, we formally state
	our main results. Section \ref{sec:configuration} is devoted to the description of the configuration model.
	In Section \ref{sec:exploration}, we describe the exploration process of Molloy and Reed and the associated distributions that govern the evolution of this 
	  process. In Section \ref{sec:existence} and Section \ref{sec:size}, we prove our main results for the supercritical case, namely when a giant component exists with high probability. In Section \ref{sec:subcritical} 
	  we prove a sublinear upper bound on the size of the largest component in the subcritical case.
	 \end{section}

\begin{section}{Definitions and preliminary concepts} \label{sec:def}
	We consider a finite simple undirected graph $\mathcal{G}  = (\mathcal{V}, \mathcal{E} )$ where $\mathcal{V}$ is the set of vertices  and $\mathcal{E}$ is the set of edges.
	We use the words ``vertices" and ``nodes" interchangeably. A \emph{path} between two vertices
	$v_1$ and $v_2$ in $\mathcal{V}$ is a collection of vertices $v_1 = u_1, u_2, \ldots, u_l = v_2$ in $\mathcal{V}$ such that for each $i = 1,2, \ldots, l-1$ we have $(u_i, u_{i+1}) \in \mathcal{E}$.
	A component, or more specifically a connected component of a graph $\mathcal{G}$ is a subgraph $\mathcal{C} \subseteq \mathcal{G}$ such that there is a path between any two vertices in $\mathcal{C}$.
	 A family of random graphs $\{ \mathcal{G}_n\}$ on $n$ vertices is said to have a giant component if there exists a positive constant $\epsilon > 0$ such that 
	 $\mathbf{P}(\mbox{There exists a component } \mathcal{C} \subseteq \mathcal{G}_n \mbox{ for which } \frac{|\mathcal{C}|}{n} \geq \epsilon) \rightarrow 1$. Subsequently, when a property holds with probability converging
	 to one as $n \rightarrow \infty$, we will say that the property hold with high probability or w.h.p. for short. 
	 
	 For any integer $p$, we use $[p]$ to denote the set $\{ 1,2, \ldots, p\}$. For any matrix $M \in \mathbbm{R}^{m \times n}$, we denote by $\|M\| \triangleq \max_{i,j} |M_{ij}|$, the largest element of the matrix $M$
	 in absolute value. It is easy to check that $\| \cdot \|$ is a valid matrix norm. We use $\delta_{ij}$ to denote the Kronecker delta function defined by 
	 \begin{align*}
	 	\delta_{ij} = \begin{cases} 1, \ \mbox{if } i = j, \\
		0, \ \mbox{otherwise}.
		\end{cases}
	 \end{align*}
	 We denote by $\mathbf{1}$ the all ones vector whose dimension will be clear from context.
	
	The notion of an asymptotic degree distribution was introduced by Molloy and Reed \cite{MolloyReed1}. In the standard unipartite case, a degree distribution dictates the fraction of vertices of a given degree. In this section we introduce an analogous notion of an asymptotic degree distribution for random multipartite graphs.
	We consider a random multipartite graph $\mathcal{G}$ on $n$ vertices with $p$ parts denoted by ${G}_1, \ldots, {G}_p$. For any $i \in [p]$ a vertex $v \in G_i$ is associated with a ``type"
	$\mathbf{d} \in \mathbbm{Z}_+^p$ which we call the ``type" of $v$. This means for each $i=1,2, \ldots , p$, the node with type $\mathbf{d}$
	has $d(i) \triangleq d_i$ neighbors in ${G}_i$. A degree distribution describes the fraction of vertices  of type $\mathbf{d}$ in $G_i, \ i \in [p]$.
	 We now define an \textit{asymptotic degree distribution} as a sequence of degree distributions which prescribe the number of vertices of type $\mathbf{d}$ in a multipartite graph on $n$ vertices.
	 For a fixed $n$, let $\mathcal{D}(n) \triangleq \left( n_{i}^{\mathbf{d}}(n), \ i \in [p], \mathbf{d} \in \{ 0,1, \ldots, n\}^p \right)$,
	  where $n_{i}^{\mathbf{d}}(n)$ denotes the number of vertices in $G_i$ of type $\mathbf{d}$. Associated with each $\mathcal{D}(n)$ is a probability 
	 distribution $\mathbf{p}(n) = \left( \frac{n_{i}^{\mathbf{d}}(n)}{n}, \ i \in [p], \mathbf{d} \in \{ 0,1, \ldots, n\}^p \right)$ which denotes the fraction of vertices of each type in each part. Accordingly, we write
	 $p_i^{\mathbf{d}}(n) = \frac{n_{i}^{\mathbf{d}}(n)}{n}$. For any vector degree $\mathbf{d}$ the quantity $\mathbf{1}' \mathbf{d}$ is simply the total degree of the vertex. We define the quantity
	 \begin{align}
	  	\omega(n) \triangleq \max \{ \mathbf{1}'\mathbf{d} : n_i^{\mathbf{d}}(n) > 0 \mbox{  for some  } i \in [p]\}, \label{maxdegreebound}
	  \end{align}
	  which is the maximum degree associated with the degree distribution $\mathcal{D}(n)$.
	  To prove our main results, we need additional assumptions on the degree sequence.
	 \begin{assumption} \label{assume}
	 	The degree sequence ${\{\mathcal{D}(n)\}}_{n \in \mathbbm{N}}$ satisfies the following conditions:
	 \begin{itemize}
	 
	 	\item[(a)] For each $n \in \mathbbm{N}$ there exists a simple graph with the  degree distribution prescribed by $\mathcal{D}(n)$, i.e., the degree sequence is a \textit{feasible} degree sequence.
	 	\item[(b)]  There exists a probability distribution $\mathbf{p} = \left( p_i^{\mathbf{d}}, \ i \in [p], \mathbf{d} \in \mathbbm{Z}_{+}^p \right)$ such that
		 the sequence of probability distributions $\mathbf{p}(n)$ associated with $\mathcal{D}(n)$ converges to the distribution $\mathbf{p}$. 
		\item[(c)] For each $i \in  [p]$, $\sum_{ \mathbf{d} } \mathbf{1}' \mathbf{d} p_i^{\mathbf{d}}(n) \rightarrow \sum_{\mathbf{d}}  \mathbf{1}' \mathbf{d} p_i^{\mathbf{d}}$.
		\item[(d)]  For each $i,j \in [p]$ such that $\lambda_i^j \triangleq \sum_{\mathbf{d}} d_j p_i^{\mathbf{d}} = 0$, the corresponding quantity 
		$\lambda_i^j(n) \triangleq \sum_{\mathbf{d}} d_j p_i^{\mathbf{d}}(n) = 0$ for all $n$.
		\item[(e)] The second moment of the degree distribution given by $\sum_{\mathbf{d}}  (\mathbf{1}' \mathbf{d})^2 p_i^{\mathbf{d}}$ exists (is finite) and  
		 		$\sum_{ \mathbf{d}} (\mathbf{1}' \mathbf{d})^2 p_i^{\mathbf{d}}(n) \rightarrow  \sum_{\mathbf{d}}  (\mathbf{1}' \mathbf{d})^2 p_i^{\mathbf{d}}$. 
	 \end{itemize}
	 \end{assumption}
	 
	 Note that the quantity $\sum_{ \mathbf{d} } \mathbf{1}' \mathbf{d} p_i^{\mathbf{d}}(n)$ in condition $(c)$ is simply $\frac{\sum_{v \in \mathcal{G}} deg(v)}{n}$. So this condition implies that the total number of edges is $O(n)$
	 , i.e., the graph is sparse. In condition $(e)$ the quantity $\sum_{ \mathbf{d}} (\mathbf{1}' \mathbf{d})^2 p_i^{\mathbf{d}}(n)$  is same as $\frac{\sum_{v \in \mathcal{G}} (deg(v))^2}{n}$. So this condition says that sum of 
	 the squares of the degrees is $O(n)$. It follows from condition (c) that $\lambda_i^j < \infty$ and that
	 $\lambda_i^j(n) \rightarrow \lambda_i^j$. The quantity  $\lambda_i^j$ is asymptotically the fraction of outgoing edges from $G_i$ to $G_j$. 
	 For $\mathbf{p}$ to be a valid degree distribution of a multipartite graph, we must have for each $1 \leq i < j \leq p$,  $\lambda_i^j = \lambda_j^i$ and for every $n$, we must have $\lambda_i^j(n) = \lambda_j^i(n)$. We have not 
	included this in the above conditions because it follows from condition (a). 
Condition $(d)$ excludes the case where there are sublinear number of edges between $G_i$ and $G_j$.

%
	
	There is an alternative way to represent some parts of Assumption \ref{assume}. For any probability distribution $\mathbf{p}$ on $\mathbbm{Z}_{+}^p$, let
	 $\mathbf{D}_{\mathbf{p}}$ denote the random variable distributed as $\mathbf{p}$. 
	 Then (b), (c) and (e) are equivalent to the following.
	\begin{itemize}
		\item[(b')] $\mathbf{D}_{\mathbf{p}(n)} \rightarrow \mathbf{D}_{\mathbf{p}}$ in distribution.
		
		\item[(c')] $\mathbf{E}[\mathbf{1}'\mathbf{D}_{\mathbf{p}(n)}] \rightarrow \mathbf{E}[\mathbf{1}'\mathbf{D}_{\mathbf{p}}]$.
		
		\item[(e')] $\mathbf{E}[(\mathbf{1}'\mathbf{D}_{\mathbf{p}(n)})^2] \rightarrow \mathbf{E}[(\mathbf{1}'\mathbf{D}_{\mathbf{p}})^2]$.	
	\end{itemize}
	
	\noindent The following preliminary lemmas follow immediately.		
	\begin{lem} \label{ui}
		 The conditions (b'), (c') and (e') together imply that the random variables $\left \{ \mathbf{1}'\mathbf{D}_{\mathbf{p}(n)} \right \}_{n \in \mathbbm{N}}$
		  and $ \left \{ \left( \mathbf{1}'\mathbf{D}_{\mathbf{p}(n)} \right)^2  \right \}_{n \in \mathbbm{N}}$ are uniformly integrable.
	\end{lem}
	\noindent Then using Lemma \ref{ui}, we prove the following statement.
	\begin{lem} \label{maxdegree}
	 	The maximum degree satisfies $\omega(n) = o(\sqrt{n})$.
	\end{lem}
	\begin{proof}
		For any $\epsilon > 0$, by Lemma \ref{ui}, there exists $q \in \mathbbm{Z}$ such that $\mathbf{E}[(\mathbf{1}'\mathbf{D}_{\mathbf{p}(n)})^2 \mathbf{1}_{\{ \mathbf{1}'\mathbf{D} > q\}}] < \epsilon  $.
		Observe that for large enough $n$, we have $\max\{  \frac{\omega^2(n)}{n}, \frac{q^2}{n}\} \leq \mathbf{E}[(\mathbf{1}'\mathbf{D}_{\mathbf{p}(n)})^2 \mathbf{1}_{\{ \mathbf{1}'\mathbf{D} > q\}}]  \leq \epsilon$.
		Since $\epsilon$ is arbitrary, the proof is complete.
	\end{proof}

		Let $S \triangleq \{(i,j) \ | \  \lambda_i^j > 0 \}$ and let $N \triangleq |S|$. For each $i \in [p]$, let $S_i \triangleq \{ j \in [p] \ | \  (i,j) \in S \}$.
		
	Note that by condition $(a)$, the set of feasible graphs with the degree distribution is non-empty. The random multipartite graph $\mathcal{G}$ we consider in this paper is drawn uniformly at random among all simple
	graphs with degree distribution given by $\mathcal{D}(n)$. The asymptotic behavior of $\mathcal{D}(n)$ is captured by the quantities $p_i^{\mathbf{d}}$. The existence of a giant component in 
	$\mathcal{G}$ as $n \rightarrow \infty$ is determined by the distribution $\mathbf{p}$.
	
\end{section}

\begin{section}{Statements of the main results}	 \label{sec:results}
	The neighborhood of a vertex in a random graph with given degree distribution resembles closely a special branching process associated with that degree distribution called the edge-biased 
	branching process. A detailed discussion of this phenomenon and results with strong guarantees for the giant component problem in random unipartite graphs can be found in \cite{BollobasRiordan12}
	and \cite{Riordan12}. 
	The edge biased branching process is defined via the edge biased degree distribution that is associated with the given degree distribution. Intuitively the edge-biased degree distribution can be 
	thought of as the degree distribution of vertices reached at the end point of an edge. Its importance will become clear when we will describe the exploration process in the sections that follow.
	We say that an edge is of type $(i,j)$ if it connects a vertex in $G_i$ with a vertex in  $G_j$. Then, as we will see, the type of the vertex in $G_j$ reached by following a random edge of type $(i,j)$ is $\mathbf{d}$ with 
	probability  $\frac{d_i p_j^{\mathbf{d}}}{\lambda_i^j}$.

	We now introduce the \emph{edge-biased branching process} which we denote by $\mathcal{T}$. Here $\mathcal{T}$ is a multidimensional branching
	process. The vertices of $\mathcal{T}$ except the root are associated with types $(i,j) \in S$. So other than the root, $\mathcal{T}$ has $N \leq p^2$ types of vertices. The root is assumed to be of a special type
	which will become clear from the description below.
	The process starts off with a root vertex $v$. With probability $p_i^{\mathbf{d}}$, the root $v$ gives rise to $d_j$ children of type $(i,j)$ for each $j \in [p]$. 	
	To describe the subsequent levels of $\mathcal{T}$ let us consider any vertex with type $(i,j)$. With probability $\frac{d_i p_j^{\mathbf{d}}}{\lambda_i^j}$ this vertex gives rise to $(d_m - \delta_{mi})$
	 children of type $(j,m)$ for each $m \in [p]$. The number of children generated by the vertices of $\mathcal{T}$ is independent for all vertices.
	  For each $n$, we define an edge-biased branching process $\mathcal{T}_n$ which we define in the same way as $\mathcal{T}$ by using the distribution 
	 $\mathcal{D}(n)$ instead of $\mathcal{D}$. We will also use the notations $\mathcal{T}(v)$ and $\mathcal{T}_n(v)$ whenever the type of the root node $v$ is specified.
	 
	 We denote the expected number of children of type $(j,m)$ generated by a vertex of type $(i,j)$ by
	\begin{align} \label{def:mu}
		\mu_{ijjm} \triangleq  \sum_{\mathbf{d}} (d_m - \delta_{im}) \frac{d_i p_j^{\mathbf{d}}}{\lambda_i^j}.
	\end{align}
	
		It is easy to see that $\mu_{ijjm} \geq 0$. Assumption 1(e) guarantees that $\mu_{ijjm} $ is finite. Note that a vertex of type $(i,j)$ cannot have children of type $(l,m)$ if $j \neq l$. But for 
		convenience we also introduce $\mu_{ijlm} = 0$ when $j \neq l$. By means of a remark we should note that  it is also possible to conduct
		 the analysis when we allow the second moments to be infinite (see for example
	\cite{MolloyReed1}, \cite{BollobasRiordan12}), but for simplicity, we do not pursue this route in this paper.
	
	Introduce a matrix $M \in \mathbbm{R}^N$ defined as follows.
	 Index the rows and columns of the matrix with double indices $(i,j) \in S$. There are $N$ such pairs denoting the $N$ rows and columns
	of $M$. The entry of $M$ corresponding to row index $(i,j)$ and column index $(l,m)$ is set to be $\mu_{ijlm}$. 
	\begin{defin}
		Let $\mathbf{A} \in \mathbbm{R}^{N \times N}$ be a matrix. Define a graph $\mathcal{H}$ on $N$ nodes where for each pair of nodes $i$ and $j$, the directed edge $(i,j)$ exists if and only if $A_{ij} > 0$.
		Then the matrix $\mathbf{A}$ is said to be \textit{irreducible} if the graph $\mathcal{H}$ is strongly connected, i.e., there exists a directed path in $\mathcal{H}$ between any two nodes in $\mathcal{H}$.
	\end{defin}
	
	We now state the well known Perron-Frobenius Theorem for non-negative irreducible matrices. This theorem has extensive applications in the study  of multidimensional branching processes
	 (see for example \cite{KestenStigum}). 
	
	\begin{thm}[Perron-Frobenius Theorem]
		Let $\mathbf{A}$ be a non-negative irreducible matrix. Then 
		\begin{itemize}
			\item[(a).] $\mathbf{A}$ has a positive eigenvalue $\gamma > 0$ such that any other eigenvalue of $\mathbf{A}$ is strictly smaller than $\gamma$ in absolute value.
			\item[(b).] There exists a left eigenvector $\mathbf{x}$ of $\mathbf{A}$ that is unique up to scalar multiplication
			 associated with the eigenvalue $\gamma$ such that all entries of $\mathbf{x}$ are positive.
		\end{itemize}
	\end{thm}
		
	We introduce the following additional assumption before we state our main results. 
	
	\begin{assumption} \label{assume1} The degree sequence ${\{\mathcal{D}(n)\}}_{n \in \mathbbm{N}}$ satisfies the following conditions.
		\begin{itemize}
			\item[(a).] The matrix $M$ associated with the degree distribution $\mathbf{p}$ is irreducible.
			\item[(b).] For each $i \in [p]$, $S_i \neq \emptyset$.
		\end{itemize}	
	\end{assumption}
	Assumption \ref{assume1} eliminates several degenerate cases. For example consider a degree distribution with $p = 4$, i.e., a $4$-partite random graph. Suppose for $i = 1,2$, we have
	$p_i^{\mathbf{d}}$ is non-zero
	only when $d_3 = d_4 = 0$, and for $i = 3,4$, $p_i^{\mathbf{d}}$ is non-zero
	only when $d_1 = d_2= 0$. In essence this distribution is associated with a random graph which is simply the union of two disjoint bipartite graphs. In particular such a graph may contain more than one
	giant component. However this is ruled out under our assumption. Further, our assumption allows us to show that the giant component has linearly many vertices in each of the $p$ parts of the multipartite graph.
	
	Let
	\begin{align} \label{survivalprobability}
		\eta \triangleq 1 - \sum_{i=1}^{\infty} \mathbf{P}(|\mathcal{T}| = i) = \mathbf{P}(|\mathcal{T}| = \infty).
	\end{align}
	Namely, $\eta$ is the survival probability of the branching process $\mathcal{T}$.
	We now state our main results.
	
	\begin{thm} \label{theorem1}
		Suppose that the Perron Frobenius eigenvalue of $M$ satisfies $\gamma > 1$. Then the following statements hold. 
		\begin{itemize}
		\item[(a)] The random graph $\mathcal{G}$ has a giant component $C \subseteq \mathcal{G}$ 
		w.h.p. Further, the size of this component $C$ satisfies
		\begin{align}
			\lim_{n \rightarrow \infty} \mathbf{P} \left(\eta - \epsilon <  \frac{|C|}{n} < \eta + \epsilon \right) = 1,
		\end{align}
		 for any $\epsilon>0$.
		\item[(b)] All components of $\mathcal{G}$ other than $C$ are of size $O(\log n)$ w.h.p.
	\end{itemize}
	\end{thm}
	
	\begin{thm} \label{theorem2}
		Suppose that the Perron Frobenius eigenvalue of $M$ satisfies $\gamma < 1$. Then all components of
		the random graph $\mathcal{G}$ are of size
		$O(\omega(n)^2\log n)$ w.h.p.
	\end{thm}
	
	The conditions of Theorem \ref{theorem1} where a giant component exists is generally referred to in the literature
	 as the supercritical case and that of Theorem \ref{theorem2} marked by the absence of a giant component is referred to as
	the subcritical case. The conditions under which giant component exists in random bipartite graphs was derived in \cite{Newmann01} using generating function heuristics. 
	We now consider the special case of a bipartite graph and show that the conditions implied by Theorem \ref{theorem1} and Theorem \ref{theorem2} reduce to that in \cite{Newmann01}.
	In this case $p=2$ and $N  = 2$. The type of all vertices $\mathbf{d}$ in $G_1$ are of the form $\mathbf{d} = (0,j)$ and those in $G_2$ are of the form $\mathbf{d} = (k,0)$.
	To match the notation in \cite{Newmann01}, we let $p_1^{\mathbf{d}} = p_j$ when $\mathbf{d} = (0,j)$ and $p_2^{\mathbf{d}} = q_k$ when $\mathbf{d} = (k,0)$. So $\lambda_1^2 = \lambda_2^1 = 
	\sum_{\mathbf{d}} d_2 p_1^{\mathbf{d}} = \sum_j j p_j = \sum_k k q_k$.
	Using the definition of $\mu_{1221}$ from equation (\ref{def:mu}), we get
	\begin{align*}
		\mu_{1221} = \sum_{\mathbf{d}} (d_1 -  \delta_{11}) \frac{d_1 p_2^{\mathbf{d}}}{\lambda_1^2} =  \frac{\sum_k k(k-1)q_k}{\lambda_1^2}.
	\end{align*}
	Similarly we can compute $\mu_{2112} =  \frac{\sum_j j(j-1)p_j}{\lambda_1^2}$. 
	From the definition of $M$, 
	\begin{align*}
		M = \left[ \begin{array}{cc} 0 & \mu_{1221} \\
							\mu_{2112} & 0	 \end{array}\right].
	\end{align*}
	The Perron-Frobenius norm of $M$ is its spectral radius
	 and is given by $(\mu_{1221}) (\mu_{2112})$. So the condition for the existence of a giant component according to Theorem \ref{theorem1} is given by $(\mu_{1221})(  \mu_{2112}) - 1 > 0$
	  which after some algebra
	reduces to 
	\begin{align*}
		\sum_{j,k} jk(jk -  j - k)p_jq_k > 0.
	\end{align*}
	This is identical to the condition mentioned in \cite{Newmann01}.
	The rest of the paper is devoted to the proof of Theorem \ref{theorem1} and Theorem \ref{theorem2}.
\end{section}

\begin{section}{Configuration Model} \label{sec:configuration}
	The configuration model \cite{Wormald78},  \cite{Bollobas}, \cite{BenderCanfield} is a convenient tool to study random graphs with given degree distributions. It provides a method to generate a multigraph from the 
	given degree distribution. When conditioned on the event that the graph is simple, the resulting distribution is uniform among all simple graphs with the given degree distribution. 
	We describe below the way to generate a configuration model from a given multipartite degree distribution.
	\begin{enumerate}
		\item For each of the $n_i^{\mathbf{d}}(n)$ vertices in $G_i$ of type $\mathbf{d}$ introduce $d_j$ clones of type $(i,j)$. An ordered pair $(i,j)$ associated with a clone designates that the clones belongs to $G_i$
		and has a neighbor in $G_j$. From the discussion following Assumption \ref{assume}, the number of clones of type $(i,j)$ is same as the number of clones of type $(j,i)$.
		\item For each pair $(i,j)$, perform a uniform random matching of the clones of type $(i,j)$ with the clones of type $(j,i)$.
		\item Collapse all the clones associated with a certain vertex back into a single vertex. This means all the edges attached with the clones of a vertex are now considered to be attached with the vertex itself.
	\end{enumerate}
	
	The following useful lemma allows us to transfer results related to the configuration model to uniformly drawn simple random graphs. 
	\begin{lem} \label{lem:config}
		If the degree sequence $\{\mathcal{D}(n)\}_{n \in \mathbbm{N}}$
		 satisfies Assumption \ref{assume}, then the probability that the configuration model results in a simple graph is bounded away from zero as $n \rightarrow \infty$.
	\end{lem}	
	As a consequence of the above lemma, any statement that holds with high probability for the random configuration model is also true with high probability for the simple random graph model. So we only need to 
	prove Theorem \ref{theorem1} and Theorem \ref{theorem2} for the configuration model.
	
	The proof of Lemma \ref{lem:config} can be obtained easily by using a similar result on directed random graphs proved in \cite{ChenMariana13}. The specifics of the proof follow.
	\begin{proof}[Proof of Lemma \ref{lem:config}]
		In the configuration model for multipartite graphs that we described, we can classify all clones into two categories. First, the clones of the kind, 
		$(i,i) \in S$ and the clones of the kind $(i,j) \in S, \ i \neq j$. Since the outcome of the matching associated with each of the cases is independent, we can treat them separately for this proof. 
		For the first category, the problem is equivalent to the case of configuration model for standard unipartite graphs. More 
		precisely, for a fixed $i$, we can construct a standard degree distribution $\tilde{\mathcal{D}}(n)$ from $\mathcal{D}(n)$ by taking the $i^{th}$ component of the corresponding vector degrees of the latter.
		By using Assumptions \ref{assume}, our proof then follows from previous results for unipartite case.
		
		For the second category, first let us fix $(i,j)$ with $i \neq j$. 
		Construct a degree distribution $\mathcal{D}_1(n) = (n^k(n), \ k \in  [n])$ where $n^k(n)$ denotes the number of vertices of degree $k$ by letting $n^k(n) = \sum_{\mathbf{d}} \mathbf{1}\{d(j) = k \} n_i^{\mathbf{d}}$.
		Construct $\mathcal{D}_2(n)$ similar to $\mathcal{D}_1(n)$ by interchanging $i$ and $j$. We consider a bipartite graph where degree distribution of the vertices in part $i$
		 is given by $\mathcal{D}_i(n)$ for $i = 1,2$. We form the corresponding configuration model and perform the usual uniform
		matching between the clones generated from  $\mathcal{D}_1(n)$ with the clones generated from
		$\mathcal{D}_2(n)$. This exactly mimics the outcome of matching that occurs in our original multipartite configuration model between clones of type $(i,j)$ and $(j,i)$. 
		With this formulation, the problem of controlling number of double edges is very closely related to a similar problem concerning the configuration model for directed random graphs which was studied in 
		\cite{ChenMariana13}. To precisely match their setting, add ``dummy" vertices with zero degree to both $\mathcal{D}_1(n)$ and $\mathcal{D}_2(n)$ so that they have exactly $n$ vertices each
		and then arbitrarily enumerate the vertices in each with indices from $[n]$.
		From Assumption \ref{assume} it can be easily verified that the degree distributions $\mathcal{D}_1(n)$ and $\mathcal{D}_2(n)$ satisfy Condition 4.2  in \cite{ChenMariana13}. To switch between our notation
		and theirs, use 
		$\mathcal{D}_1(n) \rightarrow M^{[n]}$ and $\mathcal{D}_2(n) \rightarrow D^{[n]}$. Then Theorem 4.3 in \cite{ChenMariana13} says that the probability of having no self loops and double edges is bounded
		away from zero. In particular, observing that self loops are irrelevant in our case, we conclude that $\lim_{n \rightarrow \infty} \mathbf{P}(\mbox{No double edges}) > 0$. Since the number of pairs $(i,j)$ is
		less than or equal to $p(p-1)$ which is a constant with respect to $n$, the proof is now complete.
	\end{proof}
\end{section}

\begin{section}{Exploration Process} \label{sec:exploration}
	In this section we describe the exploration process which was introduced by Molloy and Reed in \cite{MolloyReed1} to reveal the component associated with a given vertex in the random graph.
	We say a clone is of type $(i,j)$ if it belongs to a vertex in $G_i$ and has its neighbor in $G_j$. We say a vertex is of type $(i,\mathbf{d})$ if it belongs to $G_i$ and has degree type $\mathbf{d}$.
	 We start at time $k = 0$. At any point in time $k$ in the exploration process, there are three kinds of clones - `sleeping' clones , `active' clones and `dead' clones.
	 For each $(i,j) \in S$, the number of active clones of type $(i,j)$
	  at time $k$ are denoted by $A_i^ j(k)$ and the total number of active clones at time $k$ is given by
	 $A(k) = \sum_{(i, j) \in S}{A_i^j} (k)$. Two clones are said to be ``siblings" if they belong to the same vertex.
	The set of sleeping and awake clones are collectively called `living' clones. We denote by $L_i(k)$ the number of living clones in $G_i$ and $L_i^j(k)$ to be the number of living clones of type $(i,j)$ at time $k$.
	It follows that $\sum_{j \in  [p]} L_i^j(k) = L_i(k)$. 
	 If all clones of a vertex are sleeping then the vertex is said to be a sleeping vertex, if all its clones are
	dead, then the vertex is considered dead, otherwise it is considered to be active. At the beginning of the exploration process all clones (vertices) are sleeping. We denote the number of sleeping vertices 
	in $G_i$ of type $\mathbf{d}$ at time $k$
	by $N_{i}^{\mathbf{d}}(k)$ and let $N_S(k) = \sum_{i, \mathbf{d}} N_i^{\mathbf{d}}(k)$. Thus $N_{i}^{\mathbf{d}}(0) = n_i^{\mathbf{d}}(n)$ and $N_S(0) = n$.
	We now describe the exploration process used to reveal the components of the configuration model. \\ 
	
	\vspace{0.05in}
	\hspace{-0.1in} \textbf{Exploration Process.}
	\begin{itemize}
		\item[1.] \emph{Initialization}: Pick a vertex uniformly at random from the set of all sleeping vertices and and set the status 
	of all its clones to active.
	  \item[2.] Repeat the following two steps as long as there are active clones:
	\begin{itemize}
		\item[(a).] Pick a clone uniformly at random from the set of active clones and kill it.
		\item[(b).] Reveal the neighbor of the clone by picking uniformly at random one of its candidate neighbors. Kill the neighboring clone and make its siblings active.
	\end{itemize}
	\item[3.] If there are alive clones left, restart the process by picking an alive clone uniformly at random and setting all its siblings to active, and go back to step 2.
	 If there are no alive clones, the exploration process is complete.
	\end{itemize}
	\vspace{0.05in}
	 
	 Note that in step 2(b), the candidate neighbors of a clones of type $(i,j)$ are the set of alive clones of type $(j,i)$.
	 
	The exploration process enables us to conveniently track the evolution in time of the number of active clones of various types. We denote the change in  $A_i^j(k)$ by writing
	\begin{align*}
		A_i^j(k+1) = A_i^j(k) + Z_i^j(k+1), \quad  (i,j) \in S.
	\end{align*}
	Define $\mathbf{Z}(k) \triangleq \left( Z_i^j(k) , \ (i,j) \in S \right)$ to be the vector of changes in the number of active clones of all types.
	 To describe the probability distribution of the changes $Z_i^j(k+1)$, we consider the following two cases.
	 \begin{itemize}
	\item [\textit{Case 1:}]  $A(k) > 0$. \\ 
	Let $E_i^j$ denote the event that in step $2$-(a) of the exploration process, the active clone picked was of type $(i,j)$. The probability of this event is  $\frac{A_i^j(k)}{A(k)}$.
	In that case we kill the clone that we chose and the number of active clones of type $(i,j)$ reduces by one.
	Then we proceed to reveal its neighbor which of type $(j,i)$.
	One of the following events happen:
	\begin{itemize}
		\item[(i).]	$E_a$: the neighbor revealed is an active clone. The probability of the joint event is given by
				\begin{align*}
					\mathbf{P}(E_i^j \cap E_a) = \begin{cases} \frac{A_i^j(k)}{A(k)} \frac{A_j^i(k)}{L_j^i(k)} \ & \mbox{if } i \neq j, \\
					  \frac{A_i^i(k)}{A(k)} \frac{A_i^i(k) - 1}{L_i^i(k)- 1} \ & \mbox{if } i= j. \end{cases}
				\end{align*}	
				Such an edge is referred to as a back-edge in \cite{MolloyReed1}. 
				The change in active clones of different types in this joint event is as follows.
				\begin{itemize} 
					\item[-] If $i \neq j$,
					\begin{align*}
						Z_i^j(k+1) &=  Z_j^i(k+1)  = -1, \\
						Z_l^m(k+1) &= 0, \quad \mbox{otherwise }.
					\end{align*}
					\item[-] If $i = j$,
					\begin{align*}
						Z_i^i(k+1) &= -2, \\
						Z_l^m(k+1) &= 0, \quad \mbox{otherwise }.
					\end{align*}
				\end{itemize}
		
		\item[(ii).]  $E_s^{\mathbf{d}}$: The neighbor revealed is a sleeping clone of type $\mathbf{d}$.
				The probability of this joint event is given by
				\begin{align*}
					\mathbf{P}(E_i^j \cap E_s^{\mathbf{d}}) =  \frac{A_i^j(k)}{A(k)} \frac{d_i N_j^{\mathbf{d}}(k)}{L_j^i(k) - \delta_{ij}}.
				\end{align*}	
				The sleeping vertex to which the neighbor clone belongs is now active. The change in the number of active clones of different types is governed by the type $\mathbf{d}$ of this new active vertex.
				The change in active clones of different types in this event are as follows.
				\begin{itemize} 
					\item[-] If $i \neq j$,
					\begin{align*}
						Z_i^j(k+1) &= -1, \\
						Z_j^m(k+1) &= d_m - \delta_{im}, \\
						Z_l^m(k+1) &= 0, \quad \mbox{otherwise}.
					\end{align*}
					\item[-] If $i = j$,
					\begin{align*}
						Z_i^i(k+1) &=  -2 + d_i,  \\
						Z_i^m(k+1) &= d_m,  \ \mbox{for } m \neq i,\\
						Z_l^m(k+1) &= 0,   \ \mbox{otherwise }.
					\end{align*}
				\end{itemize}
				
	\end{itemize}

	Note that the above events are exhaustive, i.e., 
	\begin{align*}	
		\sum_{i,j \in S}  \sum_{\mathbf{d}} \mathbf{P}(E_i^j \cap E_s^{\mathbf{d}}) + \sum_{i,j \in S} \mathbf{P}(E_i^j \cap E_a) = 1.
	\end{align*}
	
		
	\item [\textit{Case 2:}] $A(k) = 0$. \\
	In this case, we choose a sleeping clone at random and make it and all its siblings active.
	Let $E_i^j$ be the event that the sleeping clone chosen was of type $(i,j)$. Further let $E^{\mathbf{d}}$ be the event that this clone belongs to a vertex of type $(i, \mathbf{d})$.
	Then we have 
	\begin{align*}
		\mathbf{P}(E_i^j \cap E^{\mathbf{d}}) = \frac{L_i^j(k)}{L(k)} \frac{d_j N_i^{\mathbf{d}}(k)}{L_i^j(k)} = \frac{d_j N_i^{\mathbf{d}}(k)}{L(k)}.
	\end{align*}
	In this case the change in the number of active clones of different types is given by
	\begin{align*}
		Z_i^m(k+1) &= d_m, \ \mbox{for } m \in S_i, \\
		Z_{i'}^{m'}(k+1) &= 0, \ \mbox{otherwise}.
	\end{align*}
	\end{itemize}
	We emphasize here that there are two ways in which the evolution of the exploration process deviates from that of the edge-biased branching process. First, a back-edge can occur in the exploration process
	when neighbor of an active clone is revealed to be another active clone. Second, the degree distribution of the exploration process is time dependent. 
	However, close to the beginning of the process, these two events do not have a significant impact. We exploit this fact in the following sections to prove Theorem \ref{theorem1} and \ref{theorem2}.
\end{section}

\begin{section}{Supercritical Case} \label{sec:existence}
	In this section we prove the first part of Theorem \ref{theorem1}. To do this we show that the number of active clones in the exploration process grows to a linear size with high probability.
	 Using this fact, we then prove the existence of a giant component. The idea behind the proof is as follows.
	We start the exploration process described in the previous section at an arbitrary vertex $v \in \mathcal{G}$. At the beginning of the exploration process, i.e. at $k=0$ , we 
	have $N_j^{\mathbf{d}}(0) = n p_j^{\mathbf{d}}(n)$ and $L_i^j(0) = n \lambda_i^j(n)$. So, close to the beginning of the exploration, a clone of type $(i,j)$ gives rise to
	$d_m - \delta_{im}$ clones of type $(j,m)$ with probability close to $\frac{d_i p_j^{\mathbf{d}}(n)}{\lambda_j^i(n)}$ which in turn is close to $\frac{d_i p_j^{\mathbf{d}}}{\lambda_j^i}$ for large enough $n$.
	 If we consider the exploration process in a very small linear time scale, i.e. for $k <  \epsilon n$ for small enough $\epsilon$, then the quantities $\frac{d_i N_j^{\mathbf{d}}(k)}{L_j^i(k) - \delta_{ij}}$
	  remain close to $\frac{d_i p_j^{\mathbf{d}}}{\lambda_j}$ and the quantities $\frac{A_j^i(k)}{L_j^i(k) - \delta_{ij}}$ are negligible. We use this observation to construct a process which underestimates the exploration process in some appropriate sense but 
	 whose parameters are time invariant and ``close" to the initial degree distribution. We then use this somewhat easier to analyze process to prove our result. 
	 
	We now get into the specific details of the proof. 
	We define a stochastic process $B_i^j(k)$ which we will couple with $A_i^j(k)$ such that $B_i^j(k)$ underestimates $A_i^j(k)$ with probability one. We denote the evolution in time of $B_i^j(k)$ by 
	\begin{align*}
		B_i^j(k+1) = B_i^j(k) + \hat{Z}_i^j(k+1), \quad  (i,j) \in S.
	\end{align*}
	To define  $\hat{Z}_i^j(k+1)$, we choose quantities $\pi_{ji}^{\mathbf{d}}$ satisfying
	\begin{align}
		& 0 \leq \pi_{ji}^{\mathbf{d}} < \frac{d_i p_j^{\mathbf{d}}}{\lambda_j^i} , \ \mbox{  } p_j^{\mathbf{d}} > 0, \label{underestimate} \\
		& \sum_{\mathbf{d}} \pi_{ji}^{\mathbf{d}} = 1 - \gamma, \label{undersum}
	\end{align}
	for some $0 < \gamma < 1$ to be chosen later. 
	
	We now show that in a small time frame, the parameters associated with the exploration process do not change significantly from their initial values. This is made precise in Lemma \ref{perturbation} 
	and Lemma \ref{perturbation1} below.
	Before that we first introduce some useful notation to describe these parameters for a given $n$ and at a given step $k$ in the exploration process.
	Let $M(n)$ denote the matrix of means defined analogous to $M$ by replacing $ \frac{d_i p_j^{\mathbf{d}}}{\lambda_j^i}$ by $ \frac{d_i p_j^{\mathbf{d}}(n)}{\lambda_j^i(n)}$.
	Also for a fixed $n$, define $M_k(n)$ similarly by replacing $ \frac{d_i p_j^{\mathbf{d}}}{\lambda_j^i}$ by $\frac{d_i N_j^{\mathbf{d}}(k)}{L_j^i(k) - \delta_{ij}}$. Note that $M_0(n) = M(n)$. Also from
	Assumption \ref{assume} it follows that $\frac{d_i p_j^{\mathbf{d}}(n)}{\lambda_i^j(n)} \rightarrow \frac{d_i p_j^{\mathbf{d}}}{\lambda_i^j}$ and that $M(n) \rightarrow M$. 
	
	\begin{lem} \label{perturbation}
		Given $\delta> 0$, there exists $\epsilon > 0$ and some integer $\hat{n}$ such that for all $n \geq \hat{n}$ and for all time steps $k \leq \epsilon n$ in the exploration process we have 
		$\sum_{\mathbf{d}} \left|\frac{d_i N_j^{\mathbf{d}}(k)}{L_i^j(k) - \delta_{ij}}  - \frac{d_i p_j^{\mathbf{d}}}{\lambda_i^j} \right| < \delta$.
	\end{lem}
	\begin{proof}
		Fix $\epsilon_1 > 0$.  From Lemma \ref{ui} we have that that random variables $\mathbf{1}' \mathbf{D}_{\mathbf{p}(n)}$ are uniformly integrable. Then there exists $q \in \mathbbm{Z}$ such that 
		for all $n$ we have
		$\sum_{\mathbf{d}} d_i p_j^{\mathbf{d}}(n) \mathbf{1}_{\{ \mathbf{1}' \mathbf{d} > q\}} < \epsilon_1$. Since $0 \leq \frac{N_j^{\mathbf{d}}(k)}{n} \leq \frac{N_j^{\mathbf{d}}(0)}{n} = p_j^{\mathbf{d}}(n)$, we have
		$\sum_{\mathbf{d}} \mathbf{1}_{\{ \mathbf{1}'\mathbf{d} > q\}} \left| d_i p_j^{\mathbf{d}}(n) - d_i \frac{N_j^{\mathbf{d}}(k)}{n} \right| < \epsilon_1$.
		 For each time step $k \leq \epsilon n$ in the exploration process we have $\frac{N_j^{\mathbf{d}}(k)}{n} \geq \frac{N_j^{\mathbf{d}}(0)}{n} - \epsilon$.
		 So for small enough $\epsilon$, we can make
		 $\sum_{\mathbf{d}} \mathbf{1}_{\{ \mathbf{1}'\mathbf{d} \leq q\}} \left| d_i \frac{N_j^{\mathbf{d}}(k)}{n}  - d_i p_j^{\mathbf{d}}(n) \right| < \epsilon_1$.
		 Additionally, $L_i^j(k)$ can change by at most two at each step. So $|\frac{L_i^j(k) - \delta_{ij}}{n} - \lambda_i^j(n)| \leq 2 \epsilon$. So for small enough $\epsilon$, for every $(i,j) \in S$ we have
		 $\frac{n}{L_i^j(k) - \delta_{ij}} - \frac{1}{\lambda_i^j(n)} < \epsilon_1$.
		 Now we can bound
		 \begin{align} \label{b0}
		 	\sum_{\mathbf{d}} \mathbf{1}_{\{ \mathbf{1}'\mathbf{d} > q\}} & \left|\frac{d_i N_j^{\mathbf{d}}(k)}{L_i^j(k) - \delta_{ij}}  - \frac{d_i p_j^{\mathbf{d}}(n)}{\lambda_i^j(n)} \right| \\ \notag
			&\leq \sum_{\mathbf{d}} \mathbf{1}_{\{ \mathbf{1}'\mathbf{d} > q\}} \left( \left|\frac{d_i N_j^{\mathbf{d}}(k)}{L_i^j(k) - \delta_{ij}}  - \frac{d_i N_j^{\mathbf{d}}(k)}{n \lambda_i^j(n)} \right|
			+  \left|\frac{d_i N_j^{\mathbf{d}}(k)}{n \lambda_i^j(n)}  - \frac{d_i p_j^{\mathbf{d}}(n)}{\lambda_i^j(n)} \right| \right) \\ \notag
			& \leq \sum_{\mathbf{d}} \mathbf{1}_{\{ \mathbf{1}'\mathbf{d} > q\}} \frac{d_i N_j^{\mathbf{d}}(k)}{n} \epsilon_1 + \frac{\epsilon_1}{\lambda_i^j(n)} \\ \notag
			&\leq \delta/4, \notag
		 \end{align}
		 where the last inequality can be obtained by choosing small enough $\epsilon_1$.
		 Since $q$ is a constant, by choosing small enough $\epsilon$ we can ensure that
		 $\sum_{\mathbf{d}} \mathbf{1}_{\{ \mathbf{1}'\mathbf{d} \leq q\}} \left|\frac{d_i N_j^{\mathbf{d}}(k)}{L_i^j(k) - \delta_{ij}}  - \frac{d_i p_j^{\mathbf{d}}(n)}{\lambda_i^j(n)} \right|
		 \leq \delta/4$. Additionally from Assumption \ref{assume}, for large enough $n$ we have
		 $\sum_{\mathbf{d}} \left|\frac{d_i p_j^{\mathbf{d}}(n)}{\lambda_i^j(n)} - \frac{d_i p_j^{\mathbf{d}}}{\lambda_i^j} \right| < \delta/2$.
		The lemma follows by combining the above inequalities.		
	\end{proof}
	
	\begin{lem} \label{perturbation1}
          	Given $\delta> 0$, there exists $\epsilon > 0$ and some integer $\hat{n}$ such that for all $n \geq \hat{n}$ and for all time steps $k \leq \epsilon n$ in the exploration process we have 
		$||M_k(n) - M|| \leq \delta$.
	\end{lem}
	\begin{proof}
		The argument is very similar to the proof of Lemma \ref{perturbation}. Fix $\epsilon_1 > 0$.  From Lemma \ref{ui} we know that the random variables $(\mathbf{1}' \mathbf{D}_{\mathbf{p}(n)})^2$
		 are uniformly integrable.
		It follows that there exists $q \in \mathbbm{Z}$ such that for all $n$, we have $\mathbf{E}[(\mathbf{1}' \mathcal{D}(n))^2  \mathbf{1}_{\{(\mathbf{1}' \mathcal{D}(n)) > q\}}] \leq \epsilon_1$.
		From this we can conclude that for all $i,j,m$ we have $\sum_{\mathbf{d}} (d_m - \delta_{im}) d_i p_j^{\mathbf{d}}(n)  \mathbf{1}_{\{\mathbf{1}' \mathbf{d} > q\}}\leq \epsilon_1$.
		Since $\frac{N_j^{\mathbf{d}}(0)}{n} - \epsilon \leq \frac{N_j^{\mathbf{d}}(k)}{n} \leq \frac{N_j^{\mathbf{d}}(0)}{n} = p_j^{\mathbf{d}}(n)$, we have
		\begin{align} \label{b1}
			|\sum_{\mathbf{d}} (d_m - \delta_{im}) d_i p_j^{\mathbf{d}}(n)  \mathbf{1}_{\{\mathbf{1}' \mathbf{d} > q\}} -
			 \sum_{\mathbf{d}} (d_m - \delta_{im}) \frac{d_i N_j^{\mathbf{d}}(n)}{n}  \mathbf{1}_{\{\mathbf{1}' \mathbf{d} > q\}} | \leq \epsilon_1.
		\end{align}	 
		Also $L_i^j(k)$ can change by at most $2 \epsilon n $. So, for small enough $\epsilon$, by an argument similar to the proof of Lemma \ref{perturbation}, we can prove analogous to (\ref{b0}) that 
		 \begin{align} \label{b2}
		  	\left| \sum_{\mathbf{d}} \mathbf{1}_{\{\mathbf{1}'\mathbf{d} > q\}} (d_m - \delta_{im}) \frac{d_i N_j^{\mathbf{d}}(k)}{L_i^j(k) - \delta_{ij}} 
			- \sum_{\mathbf{d}} \mathbf{1}_{\{\mathbf{1}'\mathbf{d} > q\}} (d_m - \delta_{im}) \frac{d_i p_j^{\mathbf{d}}(n)}{\lambda_i^j(n)} \right| \leq \frac{\delta}{4}.
		  \end{align}
		By choosing $\epsilon$ small enough, we can also ensure
		\begin{align} \label{b3}
			 \left|  \sum_{\mathbf{d}} \mathbf{1}_{\{\mathbf{1}'\mathbf{d} \leq q\}} (d_m - \delta_{im}) \frac{d_i N_j^{\mathbf{d}}(k)}{L_i^j(k) - \delta_{ij}} 
			  - \sum_{\mathbf{d}} \mathbf{1}_{\{\mathbf{1}'\mathbf{d} \leq q\}} (d_m - \delta_{im}) \frac{d_i p_j^{\mathbf{d}}(n)}{\lambda_i^j(n)} \right| \leq \frac{\delta}{4}.
		\end{align}	  
		 
		Since $M(n)$ converges to  $M$ we can choose $\hat{n}$ such that $||M(n) - M || \leq \frac{\delta}{2}$. By combining the last two inequalities, the proof is complete.
	\end{proof}
		
	\begin{lem} \label{valid}
		Given any $0 < \gamma < 1$, there exists $\epsilon > 0$, an integer $\hat{n} \in \mathbbm{Z}$ and quantities $\pi_{ij}^{\mathbf{d}}$ satisfying (\ref{underestimate}) and (\ref{undersum}) and the following conditions
		for all $n \geq \hat{n}$:
		\begin{itemize}
			\item[(a)] For each time step $k \leq \epsilon n$,
				\begin{align}
					\pi_{ji}^{\mathbf{d}} < \frac{d_i N_j^{\mathbf{d}}(k)}{L_j^i(k) - \delta_{ij}}, \label{coordinateunderestimate}
				\end{align}
			for each $(i,j) \in S$.
			\item[(b)] The matrix $\hat{M}$ defined analogous to $M$ by replacing $\frac{d_i p_j^{\mathbf{d}}}{\lambda_i^j}$ by $\pi_{ji}^{\mathbf{d}}$ in (\ref{def:mu}) satisfies
				\begin{align}
					||\hat{M} - M || \leq err(\gamma),	\label{muunderestimate}
				\end{align}
				where $err(\gamma)$ is a term that satisfies $\lim_{\gamma \rightarrow 0} err(\gamma) = 0$.
		\end{itemize}
	\end{lem}
	\begin{proof}
		Choose $q = q(\gamma) \in \mathbbm{Z}$ such that $\sum_{d} \frac{d_i p_j^{\mathbf{d}}}{\lambda_i^j} \mathbf{1}_{\{\mathbf{1}' \mathbf{d} > q\}} \leq \gamma/2 $. Now choose $\pi_{ji}^{\mathbf{d}}$ satisfying
		(\ref{underestimate}) and (\ref{undersum}) such that $\pi_{ji}^{\mathbf{d}} = 0$ whenever $\mathbf{1}' \mathbf{d} > q$. Using Lemma \ref{perturbation}, we can now choose $\hat{n}$ and $\epsilon$ such that 
		for every $(i,j) \in S$ and $\mathbf{d}$ such that $\mathbf{1}'\mathbf{d} \leq q$, (\ref{coordinateunderestimate}) is satisfied for all $n \geq \hat{n}$ and all $k \leq \epsilon n$.
		 The condition in part (a) is thus satisfied by this choice of $\pi_{ji}^{\mathbf{d}}$.
	
		For any $\gamma$, let us denote the choice of $\pi_{ji}^{\mathbf{d}}$ made above by $\pi_{ji}^{\mathbf{d}} (\gamma)$.  By construction, whenever $M_{ijlm} = 0$, we also have $\hat{M}_{ijlm} = 0$.
		Suppose $M_{ijjm}  = \sum_{\mathbf{d}} (d_m - \delta_{im}) \frac{d_i p_j^{\mathbf{d}}}{\lambda_i^j}  > 0$. Also, by construction we have $0 \leq \pi_{ji}^{\mathbf{d}}(\gamma) < \frac{d_i p_j^{\mathbf{d}}}{\lambda_i^j}$
		and that $\pi_{ji}^{\mathbf{d}}(\gamma) \rightarrow \frac{d_i p_j^{\mathbf{d}}}{\lambda_i^j}$ as $\gamma \rightarrow 0$.
		Let $X_{\gamma}$ be the random variable that takes the value $(d_m - \delta_{im})$ with probability $\pi_{ji}^{\mathbf{d}}(\gamma)$ and $0$ with probability $\gamma$. Similarly,
		 let $X$ be the random variable that takes the value $(d_m - \delta_{im})$
		with probability $\frac{d_i p_j^{\mathbf{d}}}{\lambda_i^j}$. Then, from the above argument have $X_{\gamma} \rightarrow X$ as $\gamma \rightarrow 0$ and that the random variable $X$ dominates the random 
		variable $X_{\gamma}$ for all $\gamma \geq 0$. Note that $X$ is integrable.
		 The proof of part (b) is now complete by using the Dominated Convergence Theorem.
		
	\end{proof}
	
	Assume that the quantities $\epsilon$ and $\pi_{ij}^{\mathbf{d}}$ have been chosen to satisfy the inequalities (\ref{coordinateunderestimate}) and (\ref{muunderestimate}).
	 We now consider each of the events that can occur at each step of the exploration process until time $\epsilon n$ and 
	describe the coupling between $Z_i^j(k+1)$ and $\hat{Z}_i^j(k+1)$ in each case.
	
	 \begin{itemize}
	\item [\textit{Case 1:}]  $A(k) > 0$. \\ 
	Suppose the event $E_i^j$ happens. We describe the coupling in case of each of the following two events.
	\begin{itemize}
		\item[(i).]	$E_a$: the neighbor revealed is an active clone. In this case we simply mimic the evolution of the number of active clones in the original exploration process. 
				Namely, $\hat{Z}_l^m(k+1) = {Z}_l^m(k+1)$ for all $l,m$.
				\item[(ii).]  $E_s^{\mathbf{d}}$: The neighbor revealed is a sleeping clone of type $\mathbf{d}$.
				In this case, we split the event further into two events $E_{s,0}^{\mathbf{d}}$ and $E_{s,1}^{\mathbf{d}}$, that is $E_{s,0}^{\mathbf{d}} \cup E_{s,1}^{\mathbf{d}} = E_s^{\mathbf{d}}$ and
				$E_{s,0}^{\mathbf{d}} \cap E_{s,1}^{\mathbf{d}} = \emptyset$. In particular,
				\begin{align*}
					\mathbf{P}(E_{s,0}^{\mathbf{d}} | E_i^j \cap E_{s}^{\mathbf{d}}) &= \frac{\pi_{ji}^{\mathbf{d}} (L_j^i(k) - \delta_{ij}) } {  d_i N_j^{\mathbf{d}}(k)  } \\
					\mathbf{P}(E_{s,1}^{\mathbf{d}} | E_i^j \cap E_{s}^{\mathbf{d}}) &= 1 -  \mathbf{P}(E_{s,0}^{\mathbf{d}} | E_i^j \cap E_{s}^{\mathbf{d}}) .
				\end{align*}
				For the above to make sense we must have $\pi_{ji} \leq \frac{d_i N_j^{\mathbf{d}}(k)}{L_j^i(k)  - \delta_{ij}} $ which is guaranteed by our choice of $\pi_{ij}^{\mathbf{d}}$.
				We describe the evolution of $B_i^j(k)$ in each of the two cases.
				\begin{itemize}
					\item[(a).] $E_{s,0}^{\mathbf{d}}$: in this case set $\hat{Z}_l^m(k+1) = {Z}_l^m(k+1)$ for all $l,m$.
									
					\item[(b).] $E_{s,1}^{\mathbf{d}}$: In this case, we mimic the evolution of the active clones of event $E_a$ instead of $E_s^{\mathbf{d}}$. More specifically,
				\begin{itemize} 
					\item[-] If $i \neq j$,
					\begin{align*}
						\hat{Z}_i^j(k+1) &=  \hat{Z}_j^i(k+1)  = -1, \\
						\hat{Z}_l^m(k+1) &= 0, \quad \mbox{otherwise }.
					\end{align*}
					\item[-] If $i = j$,
					\begin{align*}
						\hat{Z}_i^i(k+1) &= -2, \\
						\hat{Z}_l^m(k+1) &= 0, \quad \mbox{otherwise }.
					\end{align*}
				\end{itemize}
				\end{itemize}
				
	\end{itemize}
	
	\item [\textit{Case 2:}] $A(k) = 0$. \\
		Suppose that event $E_i^j \cap E^{\mathbf{d}}$ happens. In this case we split $E^{\mathbf{d}}$ into two disjoint events $E_0^{\mathbf{d}}$ and $E_1^{\mathbf{d}}$ such that 
		\begin{align*}
			\mathbf{P}(E_{0}^{\mathbf{d}} | E_i^j \cap E^{\mathbf{d}}) &= \frac{\pi_{ij}^{\mathbf{d}} ( L_j^i(k) - \delta_{ij} )} {  d_j N_i^{\mathbf{d}}(k)  } \\
			\mathbf{P}(E_{1}^{\mathbf{d}} | E_i^j \cap E^{\mathbf{d}}) &= 1 -  \mathbf{P}(E_{0}^{\mathbf{d}} | E_i^j \cap E^{\mathbf{d}}) .
		\end{align*}
		Again, the probabilities above are guaranteed to be less than one for time $k \leq \epsilon n$ because of the choice of $\pi_{ij}^{\mathbf{d}}$.
		The change in $B_i^j(k+1)$ in case of each of the above events is defined as follows.
		\begin{itemize}
			\item[(a)] $E_0^{\mathbf{d}}$.
			\begin{itemize} 
					\item[-] If $i \neq j$,
					\begin{align*}
						\hat{Z}_j^i(k+1) &= -1, \\
						\hat{Z}_i^m(k+1) &= d_m - \delta_{im}, \\
						\hat{Z}_l^m(k+1) &= 0, \quad \mbox{for } l \neq j.
					\end{align*}
					\item[-] If $i = j$,
					\begin{align*}
						\hat{Z}_i^i(k+1) &=  -2 + d_i,  \\
						\hat{Z}_i^m(k+1) &= d_m,  \ \mbox{for } m \neq i,\\
						\hat{Z}_l^m(k+1) &= 0,   \ \mbox{for } l \neq i.
					\end{align*}
				\end{itemize}
%
			\item[(b)] $E_1^{\mathbf{d}}$. 
				\begin{itemize} 
					\item[-] If $i \neq j$,
					\begin{align*}
						\hat{Z}_i^j(k+1) &=  \hat{Z}_j^i(k+1)  = -1, \\
						\hat{Z}_l^m(k+1) &= 0, \quad \mbox{otherwise }.
					\end{align*}
					\item[-] If $i = j$,
					\begin{align*}
						\hat{Z}_i^i(k+1) &= -2, \\
						\hat{Z}_l^m(k+1) &= 0, \quad \mbox{otherwise }.
					\end{align*}
				\end{itemize}
		\end{itemize}
	\end{itemize}
	This completes the description of the probability distribution of the joint evolution of the processes $A_i^j(k)$ and $B_i^j(k)$.
	
	Intuitively, we are trying to decrease the probability of the cases that actually help in the growth of the component and compensate by increasing the probability of the event which hampers the growth of the component
	(back-edges).
	From the description of the the coupling between $Z_i^j(k+1)$ and $\hat{Z}_i^j(k+1)$ it can be seen that for time $k < \epsilon n$, with 
	probability one we have  $B_i^j(k) \leq A_i^j(k)$. 
		 
		 
	Our next goal is to show
	 that for some $(i,j) \in S$ the quantity $B_i^j(k)$ grows to a linear size by time $\epsilon n$.
	   Let $H(k) = \sigma(  \{A_i^j(r), B_i^j(r), \quad (i,j) \in S, \ 1\leq r \leq k\} )$ denote the filtration of the joint exploration process till time $k$. 
	 Then the expected conditional change in $B_i^j(k)$ can be computed by considering the two cases above. First suppose that at time step $k$ we have $A(k) > 0$, i.e., we are in Case 1. We first assume that $i \neq j$.
	 Note that the only events that affect $\hat{Z}_i^j(k+1)$ are $E_i^j$ and $E_m^i$ for $m \in [p]$. Then,
	 \begin{align}
	 	\mathbf{E} [ \hat{Z}_i^j(k+1) | H(k)] &= \mathbf{P}(E_i^j | H(k)) \ \mathbf{E}[\hat{Z}_i^j(k+1) | H(k), E_i^j] \label{expectedchange} \\
			&+ \sum_{m} \mathbf{P}(E_m^i \cap E_a | H(k) )     	\ \mathbf{E}[\hat{Z}_i^j(k+1) | H(k), E_m^i \cap E_a ]  \notag \\
			&+  \sum_{m,\mathbf{d}} \mathbf{P}(E_m^i \cap E_{s0}^{\mathbf{d}} | H(k)) 	\ \mathbf{E}[\hat{Z}_i^j(k+1) | H(k), E_m^i \cap E_{s0}^{\mathbf{d}}]	 \notag \\
			&+ \sum_{m,\mathbf{d}} \mathbf{P}(E_m^i  \cap E_{s1}^{\mathbf{d}}  | H(k))\ \mathbf{E}[\hat{Z}_i^j(k+1) | H(k), E_m^i  \cap E_{s1}^{\mathbf{d}} ]. \notag
	\end{align}		
	The event $E_m^i \cap E_a$ affects $\hat{Z}_i^j(k+1)$ only when $m = j$, and in this case, $\hat{Z}_i^j(k+1)  = -1$. The same is true for the event $E_m^i   \cap E_{s1}^{\mathbf{d}}$.
	In the event $E_m^i   \cap E_{s0}^{\mathbf{d}}$, we have $ \hat{Z}_i^j(k+1) = d_j - \delta_{jm}$. Using this, the above expression is
	\begin{align*}		
			&= \frac{A_i^j(k)}{A(k)} (-1) \ + \ \frac{A_j^i(k)}{A(k)} \frac{A_i^j(k)}{L_i^j(k) } (-1) + \sum_{m, \mathbf{d}} \frac{A_m^i(k)}{A(k)} \pi_{im}^{\mathbf{d}}  (d_j - \delta_{jm})  \\
			&+  \sum_{\mathbf{d}} \frac{A_j^i(k)}{A(k)}   \left( \frac{d_j N_i^{\mathbf{d}}(k)}{L_i^j(k)}  - \pi_{ij}^{\mathbf{d}}  \right) (-1) \\
			&= \frac{A_i^j(k)}{A(k)} (-1) + \frac{A_j^i(k)}{A(k)} \frac{A_i^j(k)}{L_i^j(k)} (-1) + \sum_{m} \frac{A_m^i(k)}{A(k)} \left( \sum_{\mathbf{d}} \pi_{im}^{\mathbf{d}} (d_j - \delta_{jm}) \right) \\
			&+ \sum_{\mathbf{d}} \frac{A_j^i(k)}{A(k)}   \left( \frac{d_j N_i^{\mathbf{d}}(k)}{L_i^j(k)}  - \pi_{ij}^{\mathbf{d}}  \right) (-1). \\
	\end{align*}
	\begin{align*}
			&= \frac{A_i^j(k)}{A(k)} (-1) + \frac{A_j^i(k)}{A(k)} \left( \frac{A_i^j(k)}{L_i^j(k)} 
			+  \sum_{\mathbf{d}}   \left( \frac{d_j N_i^{\mathbf{d}}(k)}{L_i^j(k)}\right)  -  \sum_{\mathbf{d}} \pi_{ij}^{\mathbf{d}}  \right) (-1) \\
			&+  \sum_{m} \frac{A_m^i(k)}{A(k)} \left( \sum_{\mathbf{d}} \pi_{im}^{\mathbf{d}} (d_j - \delta_{jm}) \right) \\
			&\\
			&= \frac{A_i^j(k)}{A(k)} (-1) + \frac{A_j^i(k)}{A(k)} (-\gamma) +  \sum_{m} \frac{A_m^i(k)}{A(k)} \left( \sum_{\mathbf{d}} \pi_{im}^{\mathbf{d}} (d_j - \delta_{jm}) \right),
	 \end{align*}
	 where the last equality follows from (\ref{undersum}).
	 Now suppose that at time $k$ we have $A(k) = 0$, i.e., we are in Case 2. In this case, we can similarly compute
	 \begin{align*}
	 	\mathbf{E} [ \hat{Z}_i^j(k+1) | H(k)] &= \mathbf{P}(E_i^j | H(k)) \ \mathbf{E}[\hat{Z}_i^j(k+1) | H(k), E_i^j] \\
		&+  \sum_{m,\mathbf{d}} \mathbf{P}(E_i^m \cap E^{\mathbf{d}}  \cap E_{0}^{\mathbf{d}} | H(k)) 	\ \mathbf{E}[\hat{Z}_i^j(k+1) | H(k), E_i^m \cap E^{\mathbf{d}}  \cap E_{0}^{\mathbf{d}}]	\\
			&+ \sum_{m,\mathbf{d}} \mathbf{P}(E_i^m \cap  E^{\mathbf{d}}  \cap E_{1}^{\mathbf{d}}  | H(k))\ \mathbf{E}[\hat{Z}_i^j(k+1) | H(k), E_i^m \cap  E^{\mathbf{d}}  \cap E_{1}^{\mathbf{d}} ].
	 \end{align*}
	 Using the description of the coupling in Case 2, the above expression is
	 \begin{align*}
	 	&=\frac{L_j^i(k)}{L(k)} (-1) + \sum_m  \frac{L_i^m(k)}{L(k)} \sum_d \pi_{mi}^{\mathbf{d}} (d_j - \delta_{jm}) +  \sum_{\mathbf{d}} \frac{L_i^j(k)}{L(k)} \frac{d_j N_i^{\mathbf{d}}(k)}{L_i^j(k)} 
		\left(1 -  \frac{\pi_{ji}^{\mathbf{d}} L_i^j(k)}{d_j N_i^{\mathbf{d}}(k)} \right) \\
		&=\frac{L_j^i(k)}{L(k)} (-1) + \frac{L_j^i(k)}{L(k)} (-\gamma) + \sum_m  \frac{L_i^m(k)}{L(k)} \sum_d \pi_{mi}^{\mathbf{d}} (d_j - \delta_{jm}) .
	 \end{align*}
	 For the case $i = j$, a similar computation will reveal that we obtain very similar expressions to the case $i \neq j$. We give the expressions below and omit the computation.
	 For Case 1, $A(k) > 0$,
	 \begin{align*}
	 	\mathbf{E} [ \hat{Z}_i^i(k+1) | H(k)] = \frac{A_i^i(k)}{A(k)} (-1) + \frac{A_i^i(k)}{A(k)} (-\gamma) +  \sum_{m} \frac{A_m^i(k)}{A(k)} \left( \sum_{\mathbf{d}} \pi_{im}^{\mathbf{d}} (d_i - \delta_{im}) \right).
	 \end{align*}
	 and for Case 2, $A(k) = 0$,
	 \begin{align*}
	 	\mathbf{E} [ \hat{Z}_i^i(k+1) | H(k)] = \frac{L_i^i(k)}{L(k)} (-1) + \frac{L_i^i(k)}{L(k)} (-\gamma) + \sum_m  \frac{L_i^m(k)}{L(k)} \sum_d \pi_{mi}^{\mathbf{d}} (d_i - \delta_{im}) .
	 \end{align*}
	Define the vector of expected change  $\mathbf{E}[\mathbf{\hat{Z}}(k+1) | H(k)] \triangleq \left(\mathbf{E}[ Z_i^j(k+1) | H(k)] , \ (i,j) \in S \right)$. Also define $\mathbf{A}(k) = \left( \frac{A_i^j(k)}{A(k)} , \ (i,j) \in S \right)$ if $A(k) > 0$ and 
	$\mathbf{A}(k) = \left( \frac{L_i^j(k)}{L(k)} , \ (i,j) \in S \right)$ if $A(k) = 0$.
	Let $Q \in \mathbbm{R}^{N\times N}$ be given by
	\begin{align*}
		Q_{ijji} &= 1, \ \mbox{for } (i,j) \in S ,   \\
		Q_{ijlm} &= 0, \ \mbox{otherwise }.
	\end{align*}
	Then we can write the expected  change of $B_i^j(k)$ compactly as
	\begin{align} 
		\mathbf{E}[\mathbf{\hat{Z}}(k+1) | H(k)]  = \left( \hat{M}   - \gamma Q  - I  \right ) \mathbf{A}(k) . \label{eq:linearsystem}
	\end{align}
	Fix $\delta > 0$. Let $\gamma$ be small enough such that the function $err(\gamma)$ in (\ref{muunderestimate}) satisfies
	 $err(\gamma) \leq \delta$. Using Lemma $\ref{valid}$ we can choose $\epsilon$ and $\pi_{ij}^{\mathbf{d}}$ satisfying
	 (\ref{coordinateunderestimate}) and (\ref{muunderestimate}). In particular, we have $||\hat{M} - M|| \leq \delta$. For small enough $\delta$, both $M$ and $\hat{M}$ have strictly positive entries
	 in the exact same locations. Since $M$ is irreducible, it follows that $\hat{M}$ is irreducible.
	The Perron-Frobenius eigenvalue of a matrix which is the spectral norm of the matrix is a continuous function of its entries.  For small enough $\delta$, the Perron-Frobenius eigenvalue of $\hat{M}$
	 is bigger than $1$, say $1 + 2 \zeta$ for some $\zeta>0$.
	Let $\mathbf{z}$ be the corresponding left eigenvector with all positive entries and let $z_m \triangleq \min_{(i,j) \in S} z_i^j$ and $z_M \triangleq \max_{(i,j) \in S} z_i^j$.
	Define the random process $W(k) \triangleq \sum_{(i,j) \in S} z_i^j B_i^j (k)$. Then setting $\Delta W(k+1) = W(k+1) - W(k)$, from (\ref{eq:linearsystem}) we have 
	\begin{align*}
		\mathbf{E} [\Delta W(k+1) | H(k)] &= \mathbf{z}' \mathbf{E\hat{Z}}(k+1) \\
		&= \mathbf{z}' \left( \hat{M}   -   I \gamma Q  \right ) \mathbf{A}(k) \\
		&= 2 \zeta \mathbf{z}' \mathbf{A}(k)  - \gamma \mathbf{z}' Q \mathbf{A}(k).
	\end{align*}
	The first term satisfies $ 2 \zeta z_m \leq 2 \zeta \mathbf{z}' \mathbf{A}(k) \leq 2 \zeta z_M$.
	  This is because $\mathbf{1}' \mathbf{A}(k) = 1$ and hence $\mathbf{z}' \mathbf{A}(k)$ is a convex combination of the
	entries of $\mathbf{z}$. By choosing $\gamma$ small enough, we can ensure $\gamma \mathbf{z}' Q \mathbf{A}(k) \leq \zeta z_m$. Let $\kappa = \zeta z_m > 0$. Then, we have
	\begin{align} \label{positiveexpectation}
		\mathbf{E} [\Delta W(k+1) | H(k)] \geq \kappa.
	\end{align}
	
	We now use a one-sided Hoeffding bound argument to show that with high probability the quantity $W(k)$ grows to a linear size by time $\epsilon n$. Let $X(k+1) = \kappa - \Delta W(k+1)$. Then 
	\begin{align} \label{lessthan0}
		\mathbf{E}[X(k+1) | H(k)] \leq 0.
	\end{align}
	Also note that $|X(k+1)| \leq c \omega(n)$ almost surely, for some constant $c>0$.
	
	For any $B > 0$ and for  any $-B \leq x \leq B $, it can be verified that
	\begin{align*}
		e^{x} &\leq \frac{1}{2} \frac{e^B + e^{-B}}{2} + \frac{1}{2} \frac{e^B - e^{-B}}{2} x \leq e^{\frac{B^2}{2}} + \frac{1}{2} \frac{e^B - e^{-B}}{2} x.
	\end{align*}
	Using the above, we get for any $t>0$,
	\begin{align*}
		\mathbf{E}[e^{t X(k+1)} | H(k)] \leq e^{\frac{t^2 c^2 \omega^2(n)}{2}} + \frac{1}{2} \frac{e^{t c \omega(n)} - e^{-t c \omega(n)}}{2} \mathbf{E}[X(k+1)|H(k)] \leq e^{\frac{t^2 c^2 \omega^2(n)}{2}},
	\end{align*}
	where the last statement follows from (\ref{lessthan0}). We can now compute
	\begin{align*}
		\mathbf{E}[e^{t \sum_{k = 0}^{\epsilon n - 1} X(k+1)} ] = \prod_{k = 0}^{\epsilon n -1} \mathbf{E}[e^{t X(k+1)} | H(k)]  \leq e^{\frac{t^2 c^2 \omega^2(n) \epsilon n}{2}}.
	\end{align*}
	So,
	\begin{align*}
		\mathbf{P}\left( \sum_{k = 0}^{\epsilon n - 1} X(k+1) > \epsilon \kappa n /2 \right) = \mathbf{P}( e^{t \sum_{k = 0}^{\epsilon n - 1} X(k+1)  -  t  \epsilon \kappa n /2} > 1 ) \leq
		 e^{-    \frac{t \epsilon \kappa n}{2} + \frac{t^2 c^2 \omega^2(n ) \epsilon n}{2}}.
	\end{align*}
	Optimizing over $t$, we get 
	\begin{align*}	
		\mathbf{P} \left( \sum_{k = 0}^{\epsilon n - 1} X(k+1) > \epsilon \kappa n /2 \right) \leq e^{- \frac{\kappa^2  \epsilon n}{8 c^2 \omega^2(n)}  } = o(1),
	\end{align*}
	which follows by using Lemma \ref{maxdegree}.
	Substituting the definition of $X(k+1)$, 
	\begin{align}
		\mathbf{P} \left( W(\epsilon n) < \frac{\kappa \epsilon n}{2} \right) = o(1).   \label{linearw}
	\end{align}
	Recall that $W(k) = \sum_{(i,j) \in S} z_i^j B_i^j (k) \leq N z_M \max_{(i,j) \in S} B_i^j(k)  \leq N z_M \max_{(i,j) \in S} A_i^j(k)$.
	 Define $\mu \triangleq \frac{\kappa \epsilon}{2 N z_M }$. Then it follows from (\ref{linearw}) that there exists a pair $(i',j')$ such that
	 \begin{align*}
		A_{i'}^{j'}(\epsilon n) > \mu n, \quad \mbox{w.p } \quad 1 - o(1). 
	\end{align*}
		Using the fact that the number of active clones grows to a linear size we now show that the corresponding component is of linear size. To do this, we continue the exploration process in a modified fashion
	from time $\epsilon n$ onwards. By this we mean, instead of choosing active clones uniformly at random in step $2(a)$ of the exploration process, we now follow a more specific order in which we choose 
	the active clones and then reveal their neighbors. This is still a valid way of continuing the exploration process. The main technical result required for this purpose is Lemma \ref{coverall} below. 
	
	\begin{lem} \label{coverall}
		Suppose that after $\epsilon n$ steps of the exploration process, we have $A_{i'}^{j'}(\epsilon n)> \mu n$ for some pair $(i',j')$.
		 Then, there exists $\epsilon_1 > \epsilon$ and $\delta_1 > 0$ for which we can continue the exploration process in a modified way by altering the order in which active clones are chosen in step 2(a)
		  of the exploration proces such
		that at time $\epsilon_1 n$, w.h.p. for all $(i,j) \in S$, we have  $A_i^j(\epsilon_1 n) > \delta_1 n$.
	\end{lem}
	
	The above lemma says that we can get to a point in the exploration process where there are linearly many active clones of \emph{every} type. An immediate consequence of this is the Corollary \ref{coverlinear} below.
	We remark here that Corollary \ref{coverlinear} is merely one of the consequences of Lemma \ref{coverall} an can be proved in a much simpler way.
	But as we will see later, we need the full power of Lemma \ref{coverall} to prove 
	Theorem \ref{theorem1}-(b).
	
	\begin{cor} \label{coverlinear}
		Suppose that after $\epsilon n$ steps of the exploration process, we have $A_{i'}^{j'}(\epsilon n)> \mu n$ for some pair $(i',j')$. Then there exists $\delta_2 > 0$ such that w.h.p.,  
		the neighbors of the $A_{i'}^{j'}$ clones include at least $\delta_2 n $ 
		vertices in $G_j'$.
	\end{cor}
	
	Before proving Lemma \ref{coverall}, we state a well known result. The proof can be obtained by standard large deviation techniques. We omit the proof.
	
	\begin{lem} \label{revelationBound}
		Fix $m$. Suppose there are there are $n$ objects consisting of $\alpha_i n$ objects of type $i$ for $1 \leq i \leq m$.
		Let $\beta > 0$ be a constant that satisfies $\beta < \max_i \alpha_i$.
		Suppose we pick $\beta n$ objects at random from these $n$ objects without replacement. Then for given $\epsilon' > 0$ there exists $z = z(\epsilon', m)$ such that, 
		\begin{align*}
			\mathbf{P}\left( \left| \frac{\# \mbox{objects chosen of type } i}{n} - \alpha_i \beta  \right|  > \epsilon' \right) < z^n.
		\end{align*}
		
	\end{lem}
	
	\begin{proof}[Proof of Lemma \ref{coverall}]
		The proof relies on the fact that the matrix $M$ is irreducible. If we denote the underlying graph associated with $M$ by $\mathcal{H}$, then $\mathcal{H}$ is strongly connected. 
		We consider the subgraph $\mathcal{T}_{i'}^{j'}$ of 
		$\mathcal{H}$ which  is the shortest path tree in $\mathcal{H}$ rooted at the node $(i',j')$. We traverse $\mathcal{T}_{i'}^{j'}$ breadth first.  Let $d$ be the depth of $\mathcal{T}_{i'}^{j'}$.
		We continue the exploration process from this point in $d$ stages $1,2, \ldots, d$. Stage $1$ begins right after time $\epsilon n$. Denote the time at which stage $l$ ends by $\epsilon_l n$.
		 For convenience, we will assume a base stage $0$, which includes
		all events until time $\epsilon n$.
		For $1 \leq l \leq d$, let $\mathcal{I}_l$ be the set of nodes $(i,j)$ at depth $l$ in $\mathcal{T}_{i'}^{j'}$. We let $\mathcal{I}_0 = \{ (i',j')\}$.
		
		We will prove by induction that for $l = {0,1, \ldots, d}$, there exists $\delta^{(l)} > 0$ such that at the end of stage $l$, we have w.h.p.,
		$A_i^j > \delta^{(l)} n$ for each $(i,j) \in \bigcup_{x = 0}^{l} \mathcal{I}_x$.
		Note that at the end of stage $0$ we have w.h.p. $A_{i'}^{j'} > \mu n$. So we can choose $\delta^{(0)} = \mu$ to satisfy the base 
		case of the induction.
		Suppose $| \mathcal{I}_l | = r $.
		Stage $l+1$ consists of $r$ substages, namely $(l+1,1), ({l+1},2), \ldots, ({l+1},r)$ where each substage addresses exactly one $(i,j) \in \mathcal{I}_l$.
		We start stage $({l+1},1)$ by considering any $(i,j) \in \mathcal{I}_l$. We reveal the neighbors of $\alpha \delta^{(l)} n$ clones among the
		$A_{i}^{j} > \delta^{(l)} n$ clones one by one. Here $0 < \alpha < 1$ is a constant that will describe shortly. The evolution of active clones in each of these $\alpha \delta^{(l)}n$ steps is identical to 
		that in the event $E_{i}^{j}$ in Case 1 of
		 the original exploration process. 
		 Fix any $(j,m) \in \mathcal{I}_{l+1}$. Note that $M_{ijjm} > 0$ by construction of  $\mathcal{T}_{i'}^{j'}$.
		 So by making $\epsilon$ and $\epsilon_{1}, \ldots, \epsilon_{l}$ smaller if necessary and choosing $\alpha$ small enough,
		  we can conclude using Lemma \ref{perturbation1} that for all time steps $k < \epsilon_l n + \alpha \delta^{(l)}n$ we have $|| M_k(n) - M || < \delta$ for any $\delta > 0$. Similarly, by
		   using Lemma \ref{perturbation}, we get
		  \begin{align}  \label{boundonactive}
		  	\sum_{\mathbf{d}} \left( -\frac{d_i N_j^{\mathbf{d}}(k)}{L_i^j(k) - \delta_{ij}}  + \frac{d_i p_j^{\mathbf{d}}}{\lambda_i^j}\right)  =   \frac{A_i^j(k) - \delta_{ij}}{L_i^j(k) - \delta_{ij}} \leq 
			\sum_{\mathbf{d}} \left|\frac{d_i N_j^{\mathbf{d}}(k)}{L_i^j(k) - \delta_{ij}}  - \frac{d_i p_j^{\mathbf{d}}}{\lambda_i^j} \right| < \delta.
		  \end{align}
		  By referring to the description of the exploration process for the event $E_i^j$ in $Case \ 1$,
		  the expected change in $Z_j^m(k+1)$ during stage $(l+1,1)$ can be computed similar to (\ref{expectedchange}) as
		  \begin{align*}
		  	\mathbf{E}[Z_j^m(k+1) | H(k)] &= \frac{A_j^i(k) - \delta_{ij}}{L_i^j(k) - \delta_{ij}} (- \delta_{im}) +  \sum_{\mathbf{d}} \frac{d_i N_j^{\mathbf{d}}(k)}{L_i^j(k) - \delta_{ij}} (d_m  - \delta_{im}) \\
			&= \left( M_k(n)\right)_{ijjm}  - \frac{A_j^i(k) - \delta_{ij}}{L_i^j(k) - \delta_{ij}} (- \delta_{im}) \\
			&\stackrel{(a)}{\geq} M_{ijjm} - 2 \delta \stackrel{(b)}{\geq} \delta,
		  \end{align*}
		  where $(a)$ follows from (\ref{boundonactive}) and $(b)$ can be guaranteed by choosing small enough $\delta$.
		The above argument can be repeated for each $(j,m) \in \mathcal{I}_{l+1}$.
		  We now have all the ingredients we need to repeat the 
		  one-sided Hoeffding inequality argument earlier in this section.
		   We can then conclude that there exists $\delta_{j}^m > 0$ such that w.h.p. we have at least $\delta_{j}^m n$ 
		  active clones of type $(j,m)$ by the end of stage $({l+1},1)$. By the same argument, this is also true for all children of $(i, j)$ in $\mathcal{T}_{{i'}}^{{j'}}$.
		 Before starting stage $S_{l+1}^2$, we set $\delta^{(l)} = \min\{ (1 - \alpha) \delta^{(l)}, \delta_{j_1}^m  \}$. This makes sure that at every substage of stage $l$ we have at least $\delta^{(l)} n$ clones
		 of each kind that has been considered before. This enables us to use the same argument for all substages of stage $l$. 
		  By continuing in this fashion, we can conclude that at the end of stage $l+1$ we have $\delta^{(l+1)} n $ clones of each type $(i,j)$ for each $(i,j) \in \bigcup_{x = 1}^{l+1} \mathcal{I}_x$ for 
		  appropriately defined $\delta^{(l+1)}$. 
		  The proof is now complete by induction.
	\end{proof}
	
	\begin{proof}[Proof of Corollary \ref{coverlinear}]
		Consider any $j \in [p]$. We will prove that the giant component has linearly many vertices in $G_j$ with high probability. 
		
		 Let $\mathbf{d}$ be such that $p_j^{\mathbf{d}} > 0$ and let $d_i > 0$ for some $i \in [p]$. This means in the configuration model, each of these type $\mathbf{d}$ vertices have at least one clones of type $(j,i)$.
		  Continue the exploration process as in Lemma \ref{coverall}. For small enough
		 $\epsilon_1$ there are at least $n(p_j^{\mathbf{d}} - \epsilon_1)$ of type $(j,i)$ clones still unused at time $\epsilon_1 n$. From Lemma \ref{coverall}, with high probability we have at least $\delta_1 n$ clones of 
		 type $(i,j)$ at this point. Proceed by simply revealing the neighbors of each of these. Form Lemma \ref{revelationBound},
		  it follows that with high probability, we will cover at least a constant fraction of these clones which 
		 correspond to a linear number of vertices covered. Each of these vertices are in the giant component and the proof is now complete.
	\end{proof}	
			
	We now prove part(b) of Theorem \ref{theorem1}. Part (a) will be proved in the next section. 
	We use the argument by Molloy and Reed, except for the multipartite case, we will need the help of Lemma \ref{coverall} to complete the argument.
		\begin{proof}[Proof of Theorem \ref{theorem1} $(b)$]
		Consider two vertices $u,v \in \mathcal{G}$. We will upper bound the probability that $u$ lies in the component $C$, which is the component being explored at time $\epsilon n$
		 and $v$ lies in a component of size bigger than $\beta \log n$ other than $C$.
		To do so start the exploration process at $u$ and proceed till the time step $\epsilon_1 n$ in the statement of Lemma \ref{coverall}. At this time we are in the midst of revealing the component $C$. 
		But this may not be the component of $u$ because we may have restarted the exploration process using the ``Initialization step" at some time between $0$ and $\epsilon_1 n$. If
		it is not the component of $u$, then $u$ does not lie in $C$. So, let us assume that indeed we are exploring the component of $u$. At this point continue the exploration process
		in a different way 
		by switching to revealing the component of $v$. For $v$ to lie in a component of size greater than 
		$\beta \log n$, the number of active clones in the exploration process associated with the component of $v$ 
		must remain positive for each of the first $\beta \log n$ steps. At each step choices of neighbors are made uniformly at random. Also, from Lemma \ref{coverall}, 
		$C$ has at least $\delta_1 n$ active clones of each type. For the
		component of $v$ to be distinct from the component of $u$ this choice must be different from any of these active clones of the component of $u$. So it follows that 
		the probability of this event is bounded above by
		$(1 -  \delta_1)^{\beta \log n}$. For large enough $\beta$, this gives
		\begin{align*}
			\mathbf{P}(C(u) = C,  \ C(v) \neq C, \  |C(v)| > \beta \log n) = o(n^{-2}).
		\end{align*}		
		Using a union bound over all pairs of vertices $u$ and $v$ completes the proof.
	\end{proof}
	
\end{section}

\begin{section}{Size of the Giant Component} \label{sec:size}
	In this section we complete the proof of Theorem \ref{theorem1}-$(a)$ regarding the size of the giant component.
	For the unipartite case, the first result regarding the size of the giant component was obtained by Molloy and Reed \cite{MolloyReed2} by using Wormald's results 
	\cite{Wormald95} on using differential equations for random processes.
	 As with previous results for the unipartite case, we show that the size of the giant component as a fraction of $n$ is concentrated around the survival probability of the
	edge-biased branching process. We do this in two steps. First we show that the probability that a certain vertex $v$ lies in the giant component is approximately equal to the probability that the edge-biased
	branching process with $v$ as its root grows to infinity. Linearity of expectation then shows that the expected fraction of vertices in the giant component is equal to this probability.
	 We then prove a concentration result around this
	expected value to complete the proof of Theorem \ref{theorem1}. These statements are proved formally in Lemma \ref{lem:coupling}. 
	
	Before we go into the details of the proof, we first prove a lemma which is a very widely used application of Azuma's inequality.
	\begin{lem} \label{azuma}
		Let $\mathbf{X} = (X_1, X_2, \ldots, X_t)$ be a vector valued random variable and let $f(\mathbf{X})$ be a function defined on $\mathbf{X}$. Let $\mathcal{F}_k \triangleq \sigma(X_1, \ldots, X_k)$.
		Assume that
		\begin{align*}
			| \mathbf{E}(f(\mathbf{X}) | \mathcal{F}_k) - \mathbf{E}(f(\mathbf{X}) | \mathcal{F}_{k+1}) | \leq c.
		\end{align*}
		almost surely.
		Then
		\begin{align*}
			\mathbf{P}(| f( \mathbf{X}) - \mathbf{E}[f(\mathbf{X})]   | > s) \leq 2 e^{-\frac{s^2}{2tc^2}}.
		\end{align*}
	\end{lem}
	\begin{proof}
		The proof of this lemma is a standard martingale argument. We include it here for completeness. 
		Define the random variables $Y_0, \ldots,  Y_t$ as 
		\begin{align*}
			Y_k =  \mathbf{E}(f(\mathbf{X}) | \mathcal{F}_k).
		\end{align*}
		The sequence $\{Y_k\}$ is a martingale and $|Y_k - Y_{k+1}| \leq c$ almost surely. Also $Y_0 = f(\mathbf{X})$ and $Y_t = \mathbf{E}[f(\mathbf{X})]$.
		The lemma then follows by applying Azuma's inequality to the martingale sequence $\{Y_k\}$. 
	\end{proof}

	\begin{lem} \label{lem:coupling}
		Let $\epsilon > 0$ be given. Let $v \in \mathcal{G}$ be chosen uniformly at random. Then for large enough $n$, we have 
		\begin{align*}
			|\mathbf{P}(v \in C)  - \mathbf{P}(|\mathcal{T}| =  \infty)| \leq \epsilon.
		\end{align*}
	\end{lem}
	\begin{proof}
		We use a coupling argument similar to that used by Bollobas and Riordan \cite{BollobasRiordan12} where it was used to prove a similar result for ``local" properties of random graphs. 
		We couple the exploration process starting at $v$ with the branching process $\mathcal{T}_n(v)$ by trying to replicate the event in the branching 
		process as closely as often as possible. We describe the details below.
		
		The parameters of the distribution associated with $\mathcal{T}_n$ is given by $\frac{d_i p_j^{\mathbf{d}}(n)}{\lambda_i^j(n)}$. In the exploration process, at time step $k$ the corresponding parameters 
		are given by $\frac{d_i N_j^{\mathbf{d}}(k)}{L_j^i(k) - \delta_{ij}}$ (see Section \ref{sec:exploration}).
		We first show that for each of the first $\beta \log n$ steps of the exploration process, these two quantities are close to each other.
		The quantity $d_i N_j^{\mathbf{d}}(k)$ is the total number of sleeping clones at time $k$ of type $(j,i)$ in $G_j$ that belong to a vertex of type $\mathbf{d}$.
		At each step of the exploration process the total number of sleeping clones can change by at most $\omega(n)$. 
		Also $L_i^j(k)$ is the total number of living clones of type $(j,i)$ in $G_j$ and can change by at most two in each step.
		
		Then initially for all $(i,j)$ we have $L_i^j(0) = \Theta(n)$ and
		until time $\beta \log n$ it remains $\Theta(n)$.
		Therefore,
		\begin{align*}
			\sum_{i,j, \mathbf{d}}  \left|  \frac{d_i N_j^{\mathbf{d}}(k+1)}{L_j^i(k+1) - \delta_{ij}}  - \frac{d_i N_j^{\mathbf{d}}(k)}{L_j^i(k) - \delta_{ij}}    \right|
			&\leq \sum_{i,j , \mathbf{d}}  \left| \frac {d_i N_j^{\mathbf{d}}(k+1) -d_i N_j^{\mathbf{d}}(k)}  {L_j^i(k) - \delta_{ij}}  \right|  \\ 
			&+ \left| \frac{d_i N_j^{\mathbf{d}}(k+1)}{L_j^i(k) - \delta_{ij}}  - \frac{d_i N_j^{\mathbf{d}}(k+1)}{L_j^i(k+1) - \delta_{ij}}\right|. 
		\end{align*}
		From the explanation above, the first term is $O(\omega(n) /n)$ and the second term is $O(1/n)$. Recall that 
		$ \frac{d_i N_j^{\mathbf{d}}(0)}{L_j^i(0)}  = \frac{d_i p_j^{\mathbf{d}}(n)}{\lambda_j^i(n)}$. From this we can conclude by using a telescopic sum and triangle inequality that for time index $k \leq \beta \log n$, 
		\begin{align*}
			\sum_{i,j ,\mathbf{d}} & \left|  \frac{d_i N_j^{\mathbf{d}}(k)}{L_j^i(k) - \delta_{ij}}  - \frac{d_i p_j^{\mathbf{d}}(n)}{\lambda_i^j(n)} \right|  = O(k \omega(n) /n) = O(\omega(n) \log n/n).
		\end{align*}
		So the total variational distance between the distribution of the exploration process and the branching process at each of the first $\beta \log n$ steps is $O(\omega(n) \log n/n)$.
		We now describe the coupling between the branching process and the exploration process.
		For the first time step, note that
		the root of $\mathcal{T}_n$ has type $(i,\mathbf{d})$ with probability $p_i^{\mathbf{d}}$. We can couple this with the exploration process by
		letting the vertex awakened in the  ``Initialization step" of the exploration process to be of type $(i, \mathbf{d})$. 
		Since the two probabilities are the same, this step of the coupling succeeds with probability one.
		Suppose that we have defined the coupling until time $k < \beta \log n$. To describe the coupling at time step $k+1$ we need to consider the case of two events. The first is the event when the coupling has 
		succeeded until time $k$, i.e., the two processes are identical. In this case, since the total variational distance between the parameters of the two processes is $O(\omega(n) \log n/n)$ we perform a maximal 
		coupling, i.e., a coupling which fails with probability equal to the total variational distance. For our purposes, we do not need to describe the coupling at time $k+1$ in the event that the coupling has
		failed at some previous time step.
		The probability that the coupling succeeds at each  of the first $\beta \log n$ steps is at least 
		$\left( 1 - O(\omega(n)  \log n/n)\right)^{\beta \log n} = 1 - O(\omega(n) (\log n)^2/n) = 1 - o(1)$.
		We have shown that the coupling succeeds till time $\beta \log n$ with high probability. Assume that it indeed succeeds. In that case the component explored thus far is a tree.
		Therefore, at every step of the exploration 
		process a sleeping vertex is awakened because otherwise landing on an active clone will result in a cycle. This means if the branching process has survived up until this point, the corresponding exploration 
		process has also survived until this time and the component revealed has at least $\beta \log n$ vertices. 
		Hence,
		\begin{align*}
			\mathbf{P}(|C(v)| > \beta \log n) = \mathbf{P}(|\mathcal{T}_n| > \beta \log n) + o(1).
		\end{align*}		
		But Theorem \ref{theorem1} $(b)$ states that with high probability, there is only one component of size greater than $\beta \log n$, which is the giant component, i.e.,  
		\begin{align*}
			\mathbf{P}(v \in C) =  \mathbf{P}(|C(v)| > \beta \log n) + o(1) =  \mathbf{P}(|\mathcal{T}_n| > \beta \log n) + o(1).
		\end{align*}
		So, for large enough $n$, we have $|\mathbf{P}(v \in C) - \mathbf{P}(|\mathcal{T}_n| > \beta \log n)| \leq \epsilon /2$.
		The survival probability of the branching process $\mathcal{T}$ is given by
		\begin{align*}
			\mathbf{P}(|\mathcal{T}| = \infty) = 1 - \sum_{i=1}^{\infty} \mathbf{P}(|\mathcal{T}| = i).
		\end{align*}
		Choose $K$ large enough such that $|\mathbf{P}(|\mathcal{T}| \geq K) - \mathbf{P}(|\mathcal{T}| = \infty)| \leq \epsilon/4$.
		Also, since $\frac{d_i p_j^\mathbf{d}(n)}{\lambda_i^j(n)} \rightarrow \frac{d_i p_j^\mathbf{d}} {\lambda_i^j} $ for all $i,j, \mathbf{d}$, from the theory of branching processes, 
		for large enough $n$,
		\begin{align*}
			|\mathbf{P}(|\mathcal{T}_n| \geq K)  - \mathbf{P}(|\mathcal{T}| \geq K)| &\leq \epsilon/4, \\
			|\mathbf{P}(|\mathcal{T}_n| = \infty)  - \mathbf{P}(|\mathcal{T}| = \infty)| &\leq \epsilon/2.
		\end{align*}
		Since for large enougn $n$, we have $\mathbf{P}(|\mathcal{T}_n| = \infty) \leq \mathbf{P}(|\mathcal{T}_n| > \beta \log n) \leq \mathbf{P}(|\mathcal{T}_n| \geq K)$, the proof follows by combining the above statements.
	\end{proof}
	
	Now what is left is to show that the size of the giant component concentrates around its expected value.
	
	\begin{proof}[Proof of Theorem \ref{theorem1} (a) - (size of the giant component)]
	From the first two parts of Theorem \ref{theorem1}, with high probability we can categorize all the vertices of $\mathcal{G}$ into two parts, those which lie in the giant component, and those which lie in a component of 
	size smaller than $\beta \log n$, i.e., in small components. The expected value of the fraction of vertices in small components is $1 - \eta$ + o(1). We will now show that the fraction of vertices in small components
	concentrates around this mean.
	
	Recall that $cn \triangleq n \sum_{i \in [p], \mathbf{d} \in D} \mathbf{1}' \mathbf{d} \ p_i^{\mathbf{d}}$ is the number of edges in the configuration model. 
	Let us consider the random process where the edges of the configuration model are revealed one by one. Each edge corresponds to a matching between clones.
	Let $E_i \ 1\leq i \leq cn$ denote the (random) edges. Let $N_S$ denote the number of vertices in small components, i.e., in components of size smaller than $\beta \log n$. 
	We wish to apply Lemma \ref{azuma} to obtain the desired concentration result for which we need to bound $| \mathbf{E}[ N_S | E_1, \ldots, E_k] - \mathbf{E}[ N_S | E_1, \ldots, E_{k+1}] |$. 
	In the term $\mathbf{E}[ N_S | E_1, \ldots, E_{k+1}]$, let $E_{k+1}$ be the edge $(x,y)$. The expectation is taken over all possible outcomes of the rest of the edges with $E_{k+1}$ fixed to be the edge $(x,y)$.
	 In the first term $ \mathbf{E}[ N_S | E_1, \ldots, E_k]$, after $E_1, \ldots, E_k$ are revealed, 
	the expectation is taken over the rest of of the edges, which are chosen uniformly at random among all possible
	edges. All outcomes are equally likely. We construct a mapping from each possible outcome to an outcome that has $E_{k+1} = (x,y)$.
	 In particular, if the outcome contains the edge $(x,y)$ we can map it to the corresponding outcome with $E_{k+1} = (x,y)$ by simply cross-switching the positions of $(x,y)$ with the edge that occured at $k+1$. This
	 does not change the value of $N_S$ because it does not depend on the order in which the matching is revealed. On the other hand, if the outcome does not contain $(x,y)$, then we map it to 
	one of the outcomes with $E_{k+1} = (x,y)$ by switching the two edges connected to the vertices $x$ and $y$. We claim that
	switching two edges in the configuration model can change $N_S$ by at most $4 \beta \log n$. To see why observe that we can split the process of cross-switching two edges into four steps. In the first two steps we 
	delete each of the two edges one by one and in the next two steps we put them back one by one in the switched position.
	Deleting an edge can increase $N_S$ by at most $2 \beta \log n$ and can never reduce $N_S$. Adding an edge can decrease $N_S$ by at most $2 \beta \log n$ and can never increase $N_S$. So cross-switching 
	can either increase or decrease $N_S$ by at most $4 \beta \log n$. Using this we conclude
	\begin{align*}
		| \mathbf{E}[ N_S | E_1, \ldots, E_k] - \mathbf{E}[ N_S | E_1, \ldots, E_{k+1}] | \leq 4 \beta \log n.
	\end{align*}
	We now apply Lemma \ref{azuma} to obtain.
	\begin{align*}
		\mathbf{P}\left( \frac{1}{n}(N_S - (1 - \eta)) > \delta \right) < e^{-\frac{n^2 \delta^2}{8 n \beta \log n}} = o(1).
	\end{align*}
	Since with high probability, the number of vertices in the giant component is $n - N_S $, the above concentration result completes the proof.
	\end{proof}
\end{section}

\begin{section}{Subcritical Case} \label{sec:subcritical}
	In this section we prove Theorem \ref{theorem2}. The idea of the proof is quite similar to that of the supercritical case. The strategy of the proof is similar to that used in \cite{MolloyReed1}.
	More specifically, we consider the event $E_v$ 
	that a fixed vertex $v$ lies in a component of size greater than $\zeta \omega(n)^2\log n$ for some $\zeta > 0$.
	 We will show that $\mathbf{P}(E_v) =  o(n^{-1})$. Theorem \ref{theorem2} then follows by taking a union bound over 
	$v \in \mathcal{G}$.
	
	 Assume that we start the exploration process at the vertex $v$. For $v$ to lie in a component of size greater
	than $\zeta \omega(n)^2\log n$
	 the exploration process must remain positive for at least $\zeta \omega(n)^2\log n$ time steps, at each step of the exploration process, at most one vertex is new vertex is added to the component being revealed.
	This means at time $\zeta \omega(n)^2\log n$ we must have $A \left(  \zeta \omega(n)^2 \log n \right) > 0$,
	where recall that $A(k)$ denotes the total number of active clones at time $k$ of the exploration process.
	
	Let $H(k) = \sigma(  \{A_i^j(r), \quad (i,j) \in S, \ 1\leq r \leq k\} )$ denote the filtration of the exploration process till time $k$. We will assume that $A(k) > 0$ for $0 < k \leq \zeta \omega(n)^2\log n$ and upper bound 
	$\mathbf{P}(A(\zeta \omega(n)^2\log n) > 0)$.
	We first compute the expected conditional change in the number of active clones at time $k$ for $0 \leq k \leq \zeta \omega(n)^2\log n$ by splitting the outcomes into the several possible cases that affects 
	$\hat{Z}_i^j(k+1)$ as in (\ref{expectedchange}). 
	\begin{align*}
		\mathbf{E}[Z_i^j(k+1) | H(k)]  &= \mathbf{P}(E_i^j | H(k)) \ \mathbf{E}[{Z}(k+1) | H(k), E_i^j] \\
			&+ \sum_{m,\mathbf{d}} \mathbf{P}(E_m^i \cap E_a | H(k) )     	\ \mathbf{E}[{Z}(k+1) | H(k), E_m^i \cap E_a ] \\
			&+ \mathbf{P}(E_m^i \cap E_s^{\mathbf{d}}   | H(k)) 	\ \mathbf{E}[{Z}(k+1) | H(k), E_m^i \cap E_s^{\mathbf{d}}]
	\end{align*}
	\begin{align*}		
			&= \frac{A_i^j(k)}{A(k)} (-1) \ + \  \sum_m \frac{A_m^i(k)}{A(k)} \frac{A_i^m(k)}{L_i^m(k)} (-\delta_{mj}) \\		
			&+ \sum_{m, \mathbf{d}} \frac{A_m^i(k)}{A(k)} \frac{d_m N_i^{\mathbf{d}}(k)}{L_i^m(k)} (d_j - \delta_{jm}) \\
			&= - \frac{A_i^j(k)}{A(k)} -  \frac{A_j^i(k)}{A(k)} \frac{A_i^j(k)}{L_i^j(k)} + \sum_{m} \frac{A_m^i(k)}{A(k)} \sum_{\mathbf{d}} \frac{d_m N_i^{\mathbf{d}}(k)}{L_i^m(k)} (d_j - \delta_{jm}).
	\end{align*}
	
	We proceed with the proof in a similar fashion to the proof of the supercritical case.
	Let $\mathbf{E}[\mathbf{\hat{Z}}(k+1)|H(k)] = (\mathbf{E}[ Z_i^j(k+1) | H(k)] , \ (i,j) \in S)$ and define the vector quantity $\mathbf{A}(k) = \left( \frac{A_i^j(k)}{A(k)} , \  (i,j) \in S \right)$.
	Also define the matrix $Q(k) \in \mathbbm{R}^{N \times N}$ where rows and columns are indexed by double indices and for each $(i,j) \in S$, and
	\begin{align*}
		Q_{ijji}(k) &= -\frac{A_i^j(k)}{L_i^j(k) - \delta_{ij}}  , \\
		Q_{ijlm}(k) &= 0 \ \mbox{ for } (l,m) \neq (j,i).
	\end{align*}
	Then the expected change in the number of active clones of various types can be compactly written as
	\begin{align*}
		\mathbf{E}[\mathbf{\hat{Z}}(k+1)|H(k)]  = \left( M(k) - I + Q(k) \right) \mathbf{A}(k).
	\end{align*}
	
	As the exploration process proceeds, the matrix $M(k)$ changes over time. However for large enough $n$, it follows from Lemma \ref{perturbation1} that the difference between $M(k)$ and $M$ is small for 
	$0 \leq k \leq \frac{1}{2} \zeta \omega(n)^2\log n$. In particular given any $\epsilon > 0$, for large enough $n$, we have $||M(k) - M|| < \epsilon$. Also from 
	Lemma \ref{perturbation} we also have $||Q(k)|| < \epsilon$. Let $\mathbf{z}$ be the Perron-Frobenius eigenvector of $M$.
	By the assumption in Theorem \ref{theorem2}, we have
	\begin{align*}
		\mathbf{z}' M = (1 - \delta) \mathbf{z}',
	\end{align*}
	for some $0 < \delta <1$, where $(1-\delta) = \gamma$ is the Perron-Frobenius eigenvalue of $M$. Also let $z_m \triangleq \min_i z_i$ and $z_M \triangleq \max_i z_i$.
	Define the random process 
	\begin{align*}
		W(k) \triangleq  \sum_i z_i A_i (k) 
	\end{align*}
	Then the expected conditional change in $W(k)$ is given by
	\begin{align*}
		\mathbf{E}(\Delta W(k+1) | H(k)) &= \mathbf{z}' \mathbf{E \hat{Z}}(k+1) \\
		&=	\mathbf{z}'  \left( M(k) - I + Q(k) \right) \mathbf{A}(k) \\
		&= \mathbf{z}'  (M - I ) \mathbf{A}(k) + \mathbf{z}'  (M(k) - M + Q(k)) \mathbf{A}(k) \\
		&= (- \delta) \mathbf{z}' \mathbf{A}(k) + \mathbf{z}'  (M(k) - M + Q(k)) \mathbf{A}(k).
	\end{align*}
	We can choose $\epsilon$ small enough such that $\mathbf{z}'  (M(k) - M + Q(k)) < \frac{1}{2} \delta \mathbf{z}'$, where the inequality refers to element wise inequality.
	Thus 
	\begin{align*}
		\mathbf{E}(\Delta W(k) | H(k)) <  -\frac{1}{2} \delta z' \mathbf{A}(k) < - \frac{1}{2}\delta z_m \triangleq \kappa.
	\end{align*}
	We can now repeat the one-sided Hoeffding bound argument following equation (\ref{positiveexpectation}) in the supercritical case and obtain the following inequality:
	\begin{align*}
		\mathbf{P}( |{W}(\alpha) + \kappa \alpha)| > \delta) \leq 2 e^{- \frac{\delta^2}{2 \alpha \omega^2(n)}}.
	\end{align*}
	Setting $\alpha = \zeta \omega^2(n) \log n$ and $\delta = \frac{1}{2} \kappa \alpha$, we get
	\begin{align*}
		\mathbf{P} ({W}(\zeta \omega^2(n) \log n) > 0) \leq 2 e^{- \frac {\kappa^2 \zeta \log n} { 8}} = o(n^{-1}),
	\end{align*}
	for large enough $\zeta$.
	We conclude
	\begin{align*}
		\mathbf{P} (\mathcal{G} \mbox{ has a component bigger than } \zeta \omega^2(n)  \log n\ ) < \sum_{v \in \mathcal{G}} \mathbf{P} (C(v) > \zeta \log n) = o(1).
	\end{align*}
	This completes the proof of the theorem.
 \end{section}

\bibliographystyle{amsalpha}
\bibliography{Reference_GC}

\newcommand{\etalchar}[1]{$^{#1}$}
\providecommand{\bysame}{\leavevmode\hbox to3em{\hrulefill}\thinspace}
\providecommand{\MR}{\relax\ifhmode\unskip\space\fi MR }
\providecommand{\MRhref}[2]{%
  \href{http://www.ams.org/mathscinet-getitem?mr=#1}{#2}
}
\providecommand{\href}[2]{#2}
\begin{thebibliography}{MBHG06}

\bibitem[BC78]{BenderCanfield}
E.~A. Bender and E.~R. Canfield, \emph{The asymptotic number of labelled graphs
  with given degree sequences}, Journal of Combinatorial Theory \textbf{24}
  (1978), 296--307.

\bibitem[BEST04]{Boss04}
M.~Boss, H.~Elsinger, M.~Summer, and S.~Thurner, \emph{Network topology of the
  interback market}, Quantitative Finance \textbf{4} (2004), no.~6.

\bibitem[Bol85]{Bollobas}
B.~Bollob\'{a}s, \emph{{R}andom {G}raphs}, Academic Press (1985).

\bibitem[BR12]{BollobasRiordan12}
B.~Bollob\'{a}s and O.~Riordan, \emph{An old approach to the giant component
  problem}.

\bibitem[COC13]{ChenMariana13}
N.~Chen and M.~Olvera-Cravioto, \emph{Directed {R}andom {G}raphs with {G}iven
  {D}egree {D}istributions}, Arxiv.org \textbf{1207.2475} (2013).

\bibitem[ER60]{ErdosRenyi}
P.~Erd\H{o}s and A.~R\'{e}nyi, \emph{On the {E}volution of {R}andom {G}raphs},
  Magayr Tud. Akad. Mat. Kutato Int. Kozl \textbf{5} (1960), 17--61.

\bibitem[GCV{\etalchar{+}}07]{Goh07}
K.~Goh, M.~E. Cusick, D.~Valle, B.~Childs, M.~Vidal, and A.~Barabasi, \emph{The
  human disease network}, PNAS \textbf{104} (2007), no.~21.

\bibitem[HM12]{Molloy12}
H.~Hatami and M.~Molloy, \emph{The scaling window for a random graph with a
  given degree sequence}, Random Structures and Algorithms \textbf{41} (2012),
  99 -- 123.

\bibitem[Jac08]{Jackson08}
M.~O. Jackson, \emph{Social and economic networks}, Princeton University Press,
  2008.

\bibitem[JL08]{JansonLuczak}
S.~Janson and M.~Luczak, \emph{{A} new approach to the {G}iant {C}omponent
  {P}roblem}, Random Structures and Algorithms \textbf{37} (2008), no.~2,
  197--216.

\bibitem[KS66]{KestenStigum}
H.~Kesten and B.~P. Stigum, \emph{A {L}imit {T}heorem for {M}ultidimensional
  {G}alton-{W}atson {P}rocesses}, The Annals of Mathematical Statistics
  \textbf{37} (1966), no.~5, 1211 -- 1223.

\bibitem[KS08]{KangSeierstad08}
M.~Kang and T.G. Seierstad, \emph{The critical phase for random graphs with a
  given degree sequence}, Combinatorics, Probability and Computing \textbf{17}
  (2008), 67--86.

\bibitem[MBHG06]{Morrison06}
J.~L. Morrision, R.~Breitling, D.~J. Higham, and D.~R. Gilbert, \emph{A
  lock-and-key model for protein-protein interactions}, Bioinformatics
  \textbf{22} (2006), no.~16.

\bibitem[MR95]{MolloyReed1}
M.~Molloy and B.~Reed, \emph{A critical point for {R}andom {G}raphs with a
  given degree sequence}, Random Structures and Algorithms \textbf{6} (1995),
  161--180.

\bibitem[MR98]{MolloyReed2}
\bysame, \emph{The {S}ize of the {L}argest {C}omponent of a {R}andom {G}raph on
  a fixed {D}egree {S}equence}, Combinatorics, Probability and Computing
  \textbf{7} (1998), 295--306.

\bibitem[New01]{Newmann99}
M.E.J. Newmann, \emph{The structure of scentific collaboration networks}, Proc.
  Natl. Acad. Sci. USA \textbf{98} (2001).

\bibitem[NSW01]{Newmann01}
M.E.J. Newmann, S.H. Strogatz, and D.J. Watts, \emph{Random graphs with
  arbitrary degree distributions and their applications}, Phys. Rev. E
  \textbf{64} (2001), no.~026118.

\bibitem[Rio12]{Riordan12}
O.~Riordan, \emph{The phase transition in the configuration model},
  Combinatorics, Probability and Computing \textbf{21} (2012), no.~265--299.

\bibitem[Wor78]{Wormald78}
N.~C. Wormald, \emph{{S}ome {P}roblems in the {E}numeration of {L}abelled
  {G}raphs}, Ph.D. thesis, Newcastle University, 1978.

\bibitem[Wor95]{Wormald95}
\bysame, \emph{{D}ifferential {E}quations for {R}andom {P}rocesses and {R}andom
  {G}raphs}, Annals of Applied Probability \textbf{5} (1995), 1217--1235.

\bibitem[YGC{\etalchar{+}}07]{Yildrim07}
M.~Yildrim, K.~Goh, M.~E. Cusick, A.~Barabasi, and M.~Vidal, \emph{Drug-target
  network}, Nat Biotechnol \textbf{25} (2007).

\end{thebibliography}

\end{document}